\documentclass[12pt]{amsart}
\usepackage{url, 
	amssymb,setspace, mathrsfs,fontenc, comment}
\usepackage[alphabetic]{amsrefs}
\usepackage{fullpage} 
\usepackage{color}
\usepackage{tikz-cd}
\usepackage[all]{xy}
\usepackage{amsmath}

\usepackage{url, 
	amssymb,setspace, mathrsfs,fontenc}
\usepackage{amsrefs}
\usepackage{graphicx}
\usepackage[mathcal]{euscript}
\usepackage{verbatim}
\usepackage{hyperref}
\hypersetup{backref=true}
\usepackage{mathtools}
\usepackage{tikz}
\usetikzlibrary{chains}

\tikzset{node distance=2em, ch/.style={circle,draw,on chain,inner sep=2pt},chj/.style={ch,join},every path/.style={shorten >=4pt,shorten <=4pt},line width=1pt,baseline=-1ex}

\newtheorem{thm}{Theorem}
\newtheorem{lem}[thm]{Lemma}
\newtheorem{prop}[thm]{Proposition}
\newtheorem{conj}[thm]{Conjecture}
\newtheorem{cor}[thm]{Corollary}

\theoremstyle{remark}
\newtheorem{rem}[thm]{Remark}
\newtheorem{exam}[thm]{Example}

\DefineSimpleKey{bib}{myurl}
\newcommand\myurl[1]{\url{#1}}
\BibSpec{webpage}{
	+{}{\PrintAuthors} {author}
	+{,}{ \textit} {title}
	+{}{ \parenthesize} {date}
	+{,}{ \myurl} {myurl}
}
\usepackage{arydshln}

\newcommand{\nc}{\newcommand}

\nc{\ssec}{\subsection}

\nc{\on}{\operatorname}

\nc{\sE}{\mathscr{E}}
\nc{\sF}{\mathscr{F}}
\nc{\sL}{\mathscr{L}}
\nc{\sD}{\mathscr{D}}
\nc{\sA}{\mathscr{A}}

\nc{\cC}{\mathcal{C}}
\nc{\cG}{\mathcal{G}}
\nc{\cV}{\mathcal{V}}
\nc{\CB}{\mathcal{B}}
\nc {\K}{\mathcal{K}}

\nc{\cE} {\mathcal{E}}
\nc{\Kl}{\mathrm{Kl}}
\nc{\cO}{\mathcal{O}}
\nc{\cF}{\mathcal{F}}
\nc{\cZ}{\mathcal{Z}}
\nc{\bcZ}{\overline{\mathcal{Z}}}
\nc{\bcB}{\overline{\mathcal{B}}}
\nc{\cD}{\mathcal{D}}
\nc{\cDt}{\mathcal{D}^\times}
\nc{\cH}{\mathcal{H}}
\nc{\bZ}{\mathbb{Z}}
\nc{\bH}{\mathbb{H}}
\nc{\bQ}{\mathbb{Q}}
\nc{\bR}{\mathbb{R}}
\nc{\bC}{\mathbb{C}}
\nc{\bQl}{\overline{\mathbb{Q}}_\ell}
\nc{\bQlt}{\bQl^\times} 
\nc{\FG}{\mathrm{FG}}
\nc{\dR}{\mathrm{dR}}
\nc{\dv}{\dot{v}}
\nc{\du}{\dot{u}}

\nc{\uG}{\underline{G}}
\nc{\uc}{\underline{c}}
\nc{\uu}{\underline{u}}
\nc{\cU}{\mathcal{U}}
\nc{\rat}{\mathrm{rat}}
\nc{\Hyp}{\mathrm{Hyp}}
\nc{\Lie}{\mathrm{Lie}}
\nc{\ctheta}{\check{\theta}}
\nc{\nil}{\mathrm{nil}}
\nc{\bLX}{\overline{LX}}
\nc{\bOmega}{\overline{\Omega}}
\nc{\tOmega}{\widetilde{\Omega}}

\nc{\fF}{\mathfrak{F}}
\nc{\fB}{\mathfrak{B}}
\nc{\fZ}{\mathfrak{Z}}
\nc{\fx}{\mathfrak{x}}
\nc{\fy}{\mathfrak{y}}
\nc{\fb}{\mathfrak{b}}
\nc{\fk}{\mathfrak{k}}
\nc{\fI}{\mathfrak{i}}
\nc{\fj}{\mathfrak{j}}
\nc{\fg}{\mathfrak{g}}
\nc{\fu}{\mathfrak{u}}
\nc{\fl}{\mathfrak{l}}
\nc{\fn}{\mathfrak{n}}
\nc{\cP}{\mathcal{P}}
\nc{\cQ}{\mathcal{Q}}
\nc{\ft}{\mathfrak{t}}
\nc{\fz}{\mathfrak{z}}
\nc{\fc}{\mathfrak{c}}
\nc{\cfc}{\check{\mathfrak{c}}}
\nc{\fh}{\mathfrak{h}}
\nc{\fp}{\mathfrak{p}}
\nc{\cfp}{\mathring{\mathfrak{p}}}
\nc{\bone}{\mathbf{1}}
\nc{\tg}{\mathtt{g}}
\nc{\hfg}{\widehat{\fg}}
\nc{\ch}{\check{\fh}}
\nc{\hP}{\hat{P}}
\nc{\hg}{\widehat{\mathfrak{g}}}
\nc{\gO}{\mathfrak{g}[\![t]\!]}
\nc{\Ug}{\widehat{U}(\mathfrak{g})}
\nc{\dl}{/\!\!/}

\nc{\bGm}{\mathbb{G}_m}
\nc{\bGa}{\mathbb{G}_a}
\nc{\bL}{\mathbf{L}}
\nc{\bK}{\mathbf{K}}
\nc{\bJ}{\mathbf{J}}
\nc{\bI}{\mathbf{I}}
\nc{\bV}{\mathbb{V}}
\nc{\bM}{\mathbb{M}}
\nc{\bP}{\mathbb{P}}
\nc{\bA}{\mathbb{A}}
\nc{\bN}{\mathbb{N}}

\nc {\Q}{\mathrm{Q}}
\nc{\diag}{\mathrm{diag}}
\nc{\diff}{\mathrm{diff}}
\nc{\ev}{\mathrm{ev}}
\nc{\Res}{\mathrm{Res}}
\nc{\Fl}{\mathcal{F}\ell}
\nc{\Ad}{\mathrm{Ad}}
\nc{\ad}{\mathrm{ad}}
\nc{\pr}{\mathrm{pr}}
\nc{\Sl}{\mathfrak{sl}}
\nc{\gl}{\mathfrak{gl}}
\nc{\ra}{\rightarrow}
\nc{\tra}{\twoheadrightarrow}
\nc{\hra}{\hookrightarrow}
\nc{\quo}{\mathopen{ /\!/}}
\nc{\GL}{\mathrm{GL}}
\nc{\SL}{\mathrm{SL}}
\nc{\Sp}{\mathrm{Sp}}
\nc{\SO}{\mathrm{SO}}
\nc{\so}{\mathfrak{so}}
\nc{\PGL}{\mathrm{PGL}}
\nc{\Bun}{\mathrm{Bun}}
\nc{\supp}{\mathrm{supp}}
\nc{\bgamma}{\bar{\gamma}}
\nc{\ab}{\mathrm{ab}}
\nc{\td}{\mathrm{d}}
\nc{\Ht}{\mathrm{ht}}
\nc{\tX}{\tilde{X}}

\nc         {\rar}[1]       {\stackrel{#1}{\longrightarrow}}

\nc{\fa}{\mathfrak{a}}
\nc{\Hit}{\mathrm{Hit}}

\nc{\RS}{\mathrm{RS}}
\nc{\Loc}{\mathrm{Loc}}
\nc{\tLoc}{\widetilde{\mathrm{Loc}}}
\nc{\reg}{\mathrm{reg}}
\nc{\im}{\mathrm{Im}}

\nc{\tp}{\mathfrak{p}}
\nc{\cA}{\mathcal{A}}
\nc{\cY}{\mathcal{Y}}

\nc{\opp}{\mathrm{opp}}
\nc{\Ind}{\mathrm{Ind}}
\nc{\sAn}{\mathrm{can}}
\nc{\Lg}{\check{\fg}}
\nc{\cDelta}{\check{\Delta}}
\nc{\cPhi}{\check{\Phi}}
\nc{\LV}{\check{V}}
\nc{\Lh}{\check{h}}
\nc{\LG}{\check{G}}
\nc{\cT}{\check{T}}
\nc{\ct}{\check{\ft}}
\nc{\cB}{\check{B}}
\nc{\cb}{\check{\fb}}
\nc{\cN}{\check{N}}
\nc{\sN}{\mathcal{N}}
\nc{\cn}{\check{\fn}}
\nc{\Spec}{\mathrm{Spec}}
\nc{\End}{\mathrm{End}}
\nc{\crho}{\check{\rho}}
\nc{\clambda}{\check{\lambda}}

\nc{\rX}{\mathring{X}}
\nc{\ru}{\mathring{u}}

\nc{\sW}{\mathscr{W}}
\nc{\sH}{\mathscr{H}}
\nc{\sV}{\mathscr{V}}
\nc{\geom}{\mathrm{geom}}
\nc{\Irr}{\mathrm{Irr}}
\nc{\fm}{\mathfrak{m}}
\nc{\aff}{\mathrm{aff}}
\nc{\Aut}{\mathrm{Aut}}
\nc{\cJ}{\mathcal{J}}
\nc{\fs}{\mathfrak{s}}
\nc{\Stab}{\mathrm{Stab}}
\nc{\st}{\mathrm{st}}
\nc{\tw}{{\widetilde{w}}}
\nc{\gen}{\mathrm{gen}}
\nc{\genn}{\mathrm{genn}}
\nc{\sss}{\mathrm{ss}}
\nc{\fsp}{\mathfrak{sp}}
\nc{\Hom}{\mathrm{Hom}}
\nc{\bm}{\mathbf{m}}
\nc{\HG}{\mathcal{HG}}
\nc{\Gal}{\mathrm{Gal}}
\nc{\Sym}{\mathrm{Sym}}
\nc{\rank}{\mathrm{rank}}

\nc{\calX}{\mathcal{X}}
\nc{\tP}{\mathtt{P}}
\nc{\tL}{\mathtt{L}}
\nc{\tU}{\mathtt{U}}

\nc{\tW}{\widetilde{W}}
\nc{\tdb}{\tilde{b}}
\nc{\tdd}{\tilde{d}}
\nc{\tv}{\tilde{v}}
\nc{\Hk}{\on{Hk}}
\nc{\cL}{\mathcal{L}}
\nc{\talpha}{\widetilde{\alpha}}
\nc{\tQ}{{\widetilde{Q}}}
\nc{\ochi}{\overline{\chi}}
\nc{\tdelta}{\widetilde{\Delta}}
\nc{\wt}{\mathrm{wt}}
\nc{\fQ}{\mathfrak{Q}}

\nc{\Rep}{\mathrm{Rep}}
\nc{\Conn}{\mathrm{Conn}}
\nc{\Hecke}{\mathrm{Hecke}}
\nc{\Gr}{\mathrm{Gr}}
\nc{\GR}{\mathrm{GR}}
\nc{\IC}{\mathrm{IC}}
\nc{\Std}{\mathrm{Std}} 
\nc{\Db}{\mathrm{D}^{\mathrm{b}}}
\nc{\tr}{\mathrm{tr}}
\nc{\inv}{\mathrm{inv}}
\nc{\gr}{\mathrm{gr}}
\nc{\tmin}{\mathrm{min}}
\nc{\Fun}{\mathrm{Fun}~}

\nc{\bbA}{\mathbb{A}}
\nc{\mO}{\mathrm{O}}

\newcommand{\quash}[1]{}

\AtEndDocument{\bigskip{\footnotesize

\textsc{Tsao-Hsien Chen, School of Mathematics, University of Minnesota, Twin cities, Minneapolis, MN 55455 } \par
\textit{E-mail address}: \texttt{chenth@umn.edu} \par
		
\textsc{Lingfei Yi, Shanghai Center for Mathematical Sciences, Fudan University, Shanghai 200438, China} \par
\textit{E-mail address}: \texttt{yilingfei@fudan.edu.cn} \par
}}

\begin{document} 
\renewcommand{\thepart}{\Roman{part}}

\renewcommand{\partname}{\hspace*{20mm} Part}

\title{Singularities of orbit closures in 
loop spaces of symmetric varieties II: twisted cases} 
\author{Tsao-Hsien Chen and Lingfei Yi}
\date{\today} 
\maketitle
\begin{abstract}
   We study the singularities of orbit closures in loop symmetric varieties, extending our previous work \cite{CYUntwisted} to the twisted setting. We prove that the IC complexes of orbit closures are pointwise pure and satisfy a parity-vanishing property. We apply these geometric results to the study of twisted affine Lusztig--Vogan modules, establishing foundational results, including positivity properties of the twisted affine Kazhdan--Lusztig--Vogan polynomials. Along the way, we construct conical symplectic transversal slices in loop symmetric spaces. We provide applications to Langlands duality for real groups.
\end{abstract}
	
\tableofcontents

\section{Introduction}

Let $G$ be a connected reductive group over $k=\overline{\mathbb F}_p$ of characteristic $p\neq 2$, and let $\theta_0\to G$ be an involution. 
Let $X=\{g\in G|g\theta_0(g)=1\}$ be the associated symmetric space of $\theta_0$-anti-fixed points.
In our previous work \cite{CYUntwisted}, we studied the singularities of Iwahori orbit closures on the loop space $LX=X(\!(t)\!)$, extending the celebrated work of Lusztig-Vogan \cite{LV} on Borel orbit closures in $X$ to the loop space setting. Our results have found interesting applications to relative Langlands duality and to geometric Langlands on the real projective line \cite{CN}.

Motivated by connections with geometric Langlands on the twistor $\bP^1$ \cite{C,CYTempiric}, in this sequel we study the singularities of Iwahori orbit closures on the loop symmetric space
\[L^\theta X=\{\gamma\in LG|\gamma\theta(\gamma)=1\}\]
where $\theta$ is the twisted involution on the loop group $LG=G(\!(t)\!)$ given by 
$\theta(\gamma)(t)=\theta_0(\gamma(-t))$.

We show that the free $\mathbb Z[q,q^{-1}]$-module $M$, with basis indexed by equivariant irreducible local systems on Iwahori orbits in $L^\theta X$, carries a natural module structure over the affine Hecke algebra of $G$. We call $M$ the \emph{twisted affine Lusztig--Vogan module}.
 We show that the $\IC$-complexes for closures of Iwahori orbits on $L^\theta X$ give rise to a
\emph{Kazhdan-Lusztig basis}
 of $M$
with several remarkable properties.
The entries of the 
transition matrix between the 
standard basis and the Kazhdan-Lusztig basis
are polynomials 
in $q$, to be called the \emph{twisted affine Kazhdan-Lusztig-Vogan polynomials}, and 
we provide an algorithm to compute them. 
We further show that the  $\IC$-complexes of orbit closures 
are pointwise pure and satisfy a parity vanishing property
and we 
deduce from it the positivity of 
the twisted affine Kazhdan-Lusztig-Vogan polynomials.
Along the way, we generalize the construction of affine Mars-Springer slices to the twisted setting and establish several foundational results on Iwahori orbits, including a characterization of closed orbits and the placidness of orbit closures.

We deduce analogous results for the singularities of spherical orbit closures in $L^\theta X$ and give applications to relative Langlands duality for loop symmetric spaces, including the positivity of the twisted relative Kostka--Foulkes polynomials and the formality of the associated dg extension algebras.

The proofs rely on the methods developed in \cite{CYUntwisted}, but several new ideas and features arise in the twisted setting. First, some of our proofs of key results apply equally well to the untwisted setting and thus provide a more uniform approach to both settings. Second, in the twisted setting, we show that the affine Mars--Springer slices for spherical orbits are in fact symplectic varieties, in contrast to the untwisted setting, where the corresponding slices are Lagrangian. The symplectic structures on these slices, together with the parity vanishing property of IC complexes, imply that the abelian Satake category of spherical equivariant perverse sheaves on $L^\theta X$ is semisimple. This is a special feature of the twisted setting, as the corresponding Satake categories in the untwisted setting need not be semisimple.
Combining this semisimplicity with the bijection between spherical orbits on $L^\theta X$ and the so-called tempiric Langlands parameters for real reductive groups established in \cite[Theorems 1 and 2]{CYTempiric}, we obtain a version of the geometric Satake equivalence for real reductive groups. This provides an important step toward the Langlands duality for real groups proposed in \cite[Conjecture 20]{CYTempiric}.

The paper is organized as follows.
In Section \ref{orbits}, we study the parametrization of Iwahori and spherical orbits, following \cite{CYMatsuki}. In Section \ref{placidness}, we prove the placidness of orbit closures in $L^\theta X$, a key geometric property, and use it to establish the existence of a dimension theory on $L^\theta X$. In Section \ref{Transversal slices}, we construct conical transversal slices to both spherical and Iwahori orbits in $L^\theta X$ and establish several foundational properties of these slices. In particular, we show that the spherical slices are symplectic varieties.
In Section \ref{Affine Hecke modules}, we construct the twisted affine Lusztig--Vogan modules and introduce the twisted affine Kazhdan--Lusztig--Vogan polynomials. We also give an algorithm for computing these polynomials. In Section \ref{main results}, we prove pointwise purity and parity vanishing for the $\IC$-complexes of orbit closures and deduce the positivity of the twisted affine Kazhdan--Lusztig--Vogan polynomials. Finally, in Section \ref{applications}, we discuss applications to Langlands duality for real reductive groups.

\subsection*{Acknowledgements} 
T.-H.~Chen would like to thank Weiqiang Wang for useful discussions. He would also like to thank the National Center for Theoretical Sciences (NCTS) in Taipei for its hospitality, where part of this work were done.
The research of
T.-H.~Chen is supported by NSF grant DMS-2143722 and Simons Fellowships.
Lingfei Yi is supported by the Grant No. JIH1414062Y of Fudan University.

\section{Group data}\label{group}
\subsection{Symmetric varieties}\label{sym}
We recall some basic notations about symmetric varieties following  
 \cite{Springer} and  \cite{RSbook}.
Let $G$ be a  connected reductive group over $k=\overline{\mathbb F}_p$
of characteristic $p\neq 2$,
and  $\theta_0:G\to G$ an involution.
Let $T_0\subset B_0\subset G$ be
$\theta_0$-stable maximal torus and Borel subgroup.
Let $K=G^{\theta_0}$ and 
$X=\{g\in|g\theta_0(g)=1\}$ be 
the assoicated symmetric subgroup and 
symmetric space.
The group $G$-acts on $X$ by the 
$\theta_0$-conjugation: $g\cdot x=gx\theta_0(g)^{-1}$.

\subsection{Loop spaces}
Let $F=k((t)), \cO=k[[t]]$
be the formal Laurent and Taylor series rings
with variable $t$  respectively.
Let $LG$ and $L^+G$
be the loop group and arc group of $G$ with $k$-points 
$LG(k)=G(F)$ and $L^+G(k)=G(\cO)$.

Let $I_0\subset L^+G$ be the Iwahori subgroup associated to $B_0$.
Denote the associated set of simple affine roots by $\Delta_\aff$.
We fix a set of representatives of the extended Weyl group 
$\tW=N_{G(F)}(T_0(\cO))/T_0(\cO)$ in $N_{G(F)}(T_0(\cO))$ by
$\tw=t^\lambda w$, $w\in N_{G}(T_0)$.

Consider induced twisted involution on $LG$:
\[
\theta(\gamma(t))=\theta_0(\gamma(-t)).
\]
We denote by $LG^\theta=\{\gamma\in LG|\theta (\gamma)=\gamma\}$ the 
 $\theta$-fixed points and 
\[L^\theta X=(LG)^{\on{inv}\circ\theta}=\{\gamma\in LG|\gamma\theta(\gamma)=1\}\]
the $\theta$-anit-fixed points, referred to as 
 loop symmetric subgroup and loop symmetric space of $G$.
We have $LG^{\theta}(k)=G(F)^\theta$
and $L^\theta X(k)=G(F)^{\on{inv}\circ\theta}$.
 The loop group $LG$ acts on $L^\theta X$
 by the $\theta$-conjugation action 
 $g\cdot\gamma=g\gamma\theta(g)^{-1}$.
By \cite[Proposition 25]{CYMatsuki}, 
the map 
\[\tau:LG\to L^\theta X,\ \  g\to g\theta(g)^{-1}\]
induces a bijection on $k$-points
\begin{equation}\label{eq:tau}
G(F)/G(F)^\theta\cong G(F)^{\on{inv}\circ\theta}.
\end{equation}

\section{Orbits parametrization}\label{orbits}

\subsection{Spherical and Iwahori orbits}\label{ss:recollection orbits}
Assume $\theta_0(B_0^-)=\Ad_{w_1}B_0$ for some $w_1\in N_{G}(T_0)$
where $B_0^-$ is the opposite Borel.
Denote $w_2=\theta_0(w_1)w_1$.
By \cite[Theorem 13]{CYMatsuki}:

\begin{prop}\label{p:spherical orbits}
The $L^+G$-orbits on $L^\theta X$ are represented by 
$t^\lambda g_0 w_1^{-1}$,
where $\lambda=-w_1^{-1}\theta_0(\lambda)\in X_*(T)$,
$g_0$ represents one of the finitely many orbits of
the Levi subgroup $L_\lambda$ on
$A_{\lambda,0}=\{g_0\in L_\lambda\mid g_0=w_2\Ad_{w_1^{-1}}\theta_0(g_0)^{-1}(-1)^\lambda\}$
with respect to the action $h\cdot g_0=hg\Ad_{w_1^{-1}}\theta_0(h^{-1})$.
\end{prop}
We call the representatives of the above form the \emph{standard spherical representatives}.
Denote the set of $L^+G$-orbits on $L^\theta X$ by $\sN$.
For each $\nu\in\sN$, let $L^\theta X_\nu$ be the associated orbit on $L^\theta X$.
For $\nu,\nu'\in\sN$, let $\nu\leq\nu'$ if and only if
$\overline{L^\theta X}_\nu\subset\overline{L^\theta X}_{\nu'}$.

Since $T_0, T_0(\cO), I_0$ are $\theta$-stable, 
$\theta$ acts on $\tW$.
For each $\tw=t^\lambda w\in\tW$, denote $t_\tw=(\tw\theta(\tw))^{-1}\in T_0$.
By \cite[Theorem 35]{CYMatsuki}:

\begin{prop}\label{p:Iwahori orbits}
The $I_0$-orbits on $L^\theta X$ are represented by
$\tw g_0$, where $\theta_0(\tw)=\tw^{-1}\in\tW$,
$g_0$ represents one of the finitely many $T_0$-orbits on 
$T_\tw=\{g_0\in T_0\mid \Ad_w g_0=\theta_0(g_0)^{-1}t_\tw\}$
with respect to the action $h\cdot g_0=(\Ad_{w^{-1}}h)g_0\theta_0(h)^{-1}$.
\end{prop}

We call the representatives of the above form the 
\emph{standard Iwahori representatives}.
Denote the set of $I_0$-orbits on $L^\theta X$ by $\sV$.
For each $v\in\sV$, let $\cO_v$ be the associated orbit on $L^\theta X$.
We fix a set of standard Iwahori representatives 
$n_v:=\tw g_0$, $v\in\sV$.
For $v,v'\in\sV$, let $v\leq v'$ if and only if
$\overline\cO_v\subset\overline\cO_{v'}$.

\begin{rem}
   In \cite{CYMatsuki}, the authors have chosen another pair $(B,T)$ containing the split torus. 
This choice is to match the convention in \cite{NadlerMatsuki}, and to guarantee that the maximal torus $T$ is stable under 
the real conjugation associated to $\theta_0$.
Changing the Borel subgroup only changes the element $w_1$ in $\theta_0(B^-)=\Ad_{w_1}B$ where $B^-$ is the opposite Borel.
All results in \cite{CYMatsuki} 
not involving the real conjugation,
especially the first bijection in Theorem 13
and the first two bijections in Theorem 35 of \emph{loc. cit.},
are applicable for the choice $(B_0,T_0)$ without any changes.
\end{rem}

\subsection{Simple parahoric orbits and closed orbits}
In this section we collect some basic facts on Iwahori orbits.
Their proofs are the same as the corresponding results in the untwisted setting.

For a simple affine root $\alpha\in\Delta_\aff$,
let $P_\alpha$ be the associated standard simple parahoric subgroup.
Let $P_\alpha^+$ and $L_\alpha$ 
be the pro-unipotent radical and Levi subgroup of $P_\alpha$.
Denote by $s_\alpha\in\tW$ the simple reflection with respect to 
the hyperplane defined by $\alpha$ and take any lift of it, then
\[
P_\alpha=I_0\sqcup I_0 s_\alpha I_0.
\]

Denote by $w_v$ the twisted involution in $\tW$ 
represented by $n_v$, $v\in\sV$.
We associate to $v$ an involution $\psi_v=\on{Ad}_{n_v}\circ\theta$
so that $\Stab_{LG}(n_v)=(LG)^{\psi_v}$.
The stabilizer of $P_\alpha$ on $n_v$ is
$P_{\alpha,v}:=P_\alpha^{\psi_v}$.
For $\alpha\in\Delta_\aff$,
observe
\[
I_0\backslash P_\alpha \cO_v\cong 
I_0\backslash P_\alpha/P_{\alpha,v}\cong
\bP^1/H_{\alpha,v}
\]
where $H_{\alpha,v}\subset\Aut(\bP^1)\cong\mathrm{PGL}_2$ 
is the image of the right action of $P_{\alpha,v}$ 
on $I_0\backslash P_\alpha\cong\bP^1$.

Denote by $\phi:P_\alpha\rightarrow\Aut(\bP^1)$ the action.
Let $U_{\pm\alpha}\subset P_\alpha$ 
be the affine root subgroups for $\pm\alpha$.
Choose isomorphisms $u_{\pm\alpha}:k\xrightarrow{\sim}U_{\pm\alpha}$
as in \cite[\S4.1.1]{MS},
so that $n_\alpha=u_\alpha(1)u_{-\alpha}(-1)u_\alpha(1)$ 
is a representative of $s_\alpha$.
Observe that by looking at the image under $\tau$, we can see
$P_\alpha\cO_v$ is a union of finitely many $I$-orbits,
so that $\bP^1/H_{\alpha,v}$ is a finite set.
In particular, $H_{\alpha,v}$ is infinite.

The group $H_{\alpha,v}$ is classified as in 
\cite[Lemma 28]{CYUntwisted},
labeled by type I, IIa, IIb, IIIa, IIIb, IVa, and IVb.
We say an orbit $\cO_v$ is of type I, IIa,... for $s_\alpha$
if $H_{\alpha,v}$ is of type I, IIa,....
Then $\cO_v$ is open in $P_\alpha\cO_v$
if and only if $\cO_v$ is of type I, IIb, IIIb, or IVb for $s_\alpha$.

The involution $\psi_v=\Ad_{n_v}\circ\theta$ preserves $T_0$,
thus acts on the set of affine roots.
Since $\theta(n_v)=n_v^{-1}$,
the action of $\psi_v$ on the affine roots is given by
$w_v\circ\theta=\theta\circ w_v^{-1}$.
The type of $\cO_v$ for $s_\alpha$ can be read off
from \cite[Lemma 29]{CYUntwisted}.

Denote by $\Omega\subset\widetilde{W}$ the elements of length zero.
The closed orbits are characterized by \cite[Lemma 31]{CYUntwisted}
with the same proof replacing $\tau(I_0x_v)$ with $\cO_v$.

\section{Placidness of orbit closures}\label{placidness}
We show that the spherical orbits and Iwahori orbits on the loop symmetric space 
$L^\theta X$ are placid schemes. 
We recall the following basic placidness results for loop spaces due to Drinfeld:

\begin{thm}\label{placid of loop spaces}
Let $Y$ be a smooth affine scheme of $F$. The loop space 
$LY$ of $Y$ is a placid ind-scheme.

\end{thm}
\begin{proof}
    This is \cite[Theorem 6.3]{D}.
\end{proof}

\begin{thm}\label{Placidness of orbits}
(i)
$L^\theta X$ is a placid ind-scheme.

(ii) The orbits closures 
$\overline\cO_v$ (resp. 
$\overline {L^\theta X}_\nu$) are irreducible placid schemes.
 
(iii) The open embedding $\cO_v\to\overline\cO_v$ (resp. $L^\theta X_\nu\to\overline{L^\theta X}_\nu$) and the closed embedding
$\overline\cO_v\to \overline\cO_{v'}$, $v\leq v'$
(resp. $\overline{L^\theta X}_\nu\to\overline{L^\theta X}_{\nu'}$, $\nu\leq\nu'$)
are finitely presented.
\end{thm}
\begin{proof}
 Let $\on{Res}_{F/F'}G_F$
be the Weil restriction of $G_F=G\times_k F$
along the quadratic extension $F'=k((t^2))\subset F$. 
Consider the involution $\on{inv}\circ\theta$
on $\on{Res}_{F/F'}G_F$.
Since 
the characteristic of $F'$ is not equal to $2$,  
the fixed points 
$Y=(\on{Res}_{F/F'}G_F)^{\on{inv}\circ\theta}$
is an affine smooth $F'$-scheme and 
it follows from the definition that there is an isomorphism
\[L^\theta X\cong LY\]
between the loop symmetric space 
$L^\theta X$ and  the loop space 
of $Y$.
Now the desired claim follows from Theorem \ref{placid of loop spaces}. 
Part (i) follows.

Let $\nu\in\sN$ be represented by 
the standard spherical representative $t^\lambda g_0$.
Then we have  natural inclusions
\[L^\theta X_\nu\to (\overline{LG_\lambda})^{\on{inv}\circ\theta}\to L^\theta X=LG^{\on{inv}\circ\theta}\]
Note that the second map 
\[(\overline{LG_\lambda})^{\on{inv}\circ\theta}\to L^\theta X\]
is finitely presented 
since it is the base change of the 
finitely presented embedding
$\overline{LG_\lambda}=\overline{L^+Gt^\lambda L^+G}\to LG$
along the embedding $L^\theta X\to LG$.
Since $L^\theta X$ is ind-placid by (i), it follows that 
$(\overline{LG_\lambda})^{\on{inv}\circ\theta}$
is a plaicd scheme.
We claim that the first inclusion 
\begin{equation}\label{fp}
    L^\theta X_\nu\to (\overline{LG_\lambda})^{\on{inv}\circ\theta}
\end{equation}
is also finitely presented. 
Then
the Noetherian descent for fp-morphisms 
implies that 
there is a unique decomposition 
\[
L^\theta X_\nu\stackrel{j}\to \overline{L^\theta X_\nu}\stackrel{i}\to (\overline{LG_\lambda})^{\on{inv}\circ\theta}
\]
where $j$ is a placid fp-open embedding and $i$ is a placid fp-closed embedding 
(see, e.g., \cite[Lemma 1.3.6 and Lemma 1.3.11]{BKV}).
Part (ii) and (iii) for spherical orbits follow.
For the case of Iwahori orbits, it suffices to show that 
the inclusion $\cO_v\to L^\theta X_\nu$ 
from an Iwahori orbit to the corresponding spherical orbit  
is finitely presented.
This follows from the same proof as in \cite[Theorem 35]{CYUntwisted}. This finishes the proof of (ii) and (iii).

Proof of the claim. We have natural factorization
$L^\theta X_\nu\to (LG_\lambda)^{\on{inv}\circ\theta}\subset (\overline{LG_\lambda})^{\on{inv}\circ\theta}$
where the second inclusion is fp. Thus to show~\eqref{fp} is fp it suffices to check 
\begin{equation}\label{fp 2}
    L^\theta X_\nu\to (LG_\lambda)^{\on{inv}\circ\theta}
\end{equation}
is fp. 
Following \cite[Section 3.1.1]{CYMatsuki},
let  $L_\lambda=Z_G(\lambda)$ be the Levi centralizer of $\lambda$
and  $A_{\lambda,0}=\{g\in L_\lambda|g=w_2\Ad_{w_1^{-1}}\theta_0(g)^{-1}\lambda(-1)\}$
with twisted $L_\lambda$-conjugation action 
$h\cdot g=hg\Ad_{w_1^{-1}}\theta_0(h)^{-1}$.
It follows from \cite[Proposition 9]{CYMatsuki} that there are Cartesian diagrams
\[
\xymatrix{L^\theta X_\nu\ar[r]\ar[d]&L^+G\backslash L^\theta X_\nu\ar[r]\ar[d]&L_\lambda\backslash\cO_\nu\ar[d]^i\\
(LG_\lambda)^{\on{inv}\circ\theta}\ar[r]&L^+G\backslash (LG_\lambda)^{\on{inv}\circ\theta}\ar[r]^{\ \ \ f}&L_\lambda\backslash\ A_{\lambda,0}}
\]
where $i:\cO_\nu\to A_{\lambda,0}$ is an embedding of 
an $L_\lambda$-orbit and the map $f$ is induced by the 
evaluation map $L^+G\to G, \gamma(t)\to \gamma(0)$:
\[f:L^+G\backslash (LG_\lambda)^{\on{inv}\circ\theta}\cong 
L^+G\cap\Ad_{t^\lambda}L^+G\backslash (t^\lambda L^+G)^{\on{inv}\circ\theta}\to L_\lambda\backslash A_{\lambda,0},\ \ t^\lambda\gamma(t)\to \gamma(0).\]
Since $i$ is fp, it follows that~\eqref{fp 2} is fp.
The claim follows.
\end{proof}

\begin{rem}
The argument in Theorem \ref{Placidness of orbits} works equally well in the untwisted setting $LX$ and hence provides an alternative proof of \cite[Theorem 35]{CYUntwisted}. That being said, the methods of \emph{loc. cit.} have the advantage of being applicable to general spherical varieties, see \cite[Remark 37]{CYUntwisted}.
\end{rem}

\subsection{Filtered structures on orbits}
Consider the partial order on the sets of 
spherical orbits $\sN$ and Iwahori orbits $\sV$
on $L^\theta X$ by inclusion of orbit closures.

\begin{lem}\label{filtered}
(i) On each connected component of $L^\theta X$, the partial order on the spherical orbits is filtered.

(ii) On each connected component of $L^\theta X$, the partial order on the Iwahori orbits is filtered.
\end{lem}
\begin{proof}
Let $L^\theta X_\nu$ and $L^\theta X_{\nu'}$
be two spherical orbits in the same component of $L^\theta X$.
Since $\tau:LG\to L^\theta X$ induces a surjection 
$LG(k)\to LG(k)/LG(k)^\theta\cong L^\theta (X)(k)$ on $k$-points, we have a surjection on the component groups
\[
\pi_0(LG)\to \pi_0(L^\theta X)\to 1.
\]
It follows that we can find an element 
$\gamma$ in the identity component of $LG$
such that $\gamma\cdot L^\theta X_\nu\cap L^\theta X_{\nu'}$ is non-empty.
We have $\gamma\in LG_\lambda$ for some dominant co-weight $\lambda$. Since $\lambda$ is in the identity component of $LG$, we have 
$L^+G=LG_0\subset\overline{LG_\lambda}$.
Consider the Hecke action
\[a:\overline{LG_\lambda}\times^{L^+G}\overline{L^\theta X_\nu}\to L^\theta X.\]
Since $\overline{LG_\lambda}$ contains both 
$\gamma$ and $L^+G$,
it follows that $L^\theta X_\nu,L^\theta X_{\nu'}\subset\on{Im}(a)$.
On the other hand, since $\on{Im}(a)$ is a proper
closed $L^+G$-invariant irreducible subset of $L^\theta X$, it contains only finitely many $L^+G$-orbits and we conclude that 
\[\on{Im}(a)=\overline{L^\theta X_{\nu''}}\]
for some $\nu''$. Part (i) follows.

By \cite{CYMatsuki}, each spherical orbit  contains only finitely many Iwahori orbits and we can follow the same argument as in \cite[Lemma 38]{CYUntwisted} to conclude part (ii).
\end{proof}

Recall the notion of locally equidimensional  placid presentations of ind-schemes in \cite[Appendix A]{CYUntwisted}.
Consider the partially ordered sets  $\underline\sV=\{\underline v=\{v_\beta\}_{\beta\in\pi_0(L^\theta X)}|\cO_{v_\beta}\in (L^\theta X)_\beta\}$
and $\underline{\sN}=\{\underline\nu=\{\nu_\beta\}_{\beta\in\pi_0(L^\theta X)}|L^\theta X_{\nu_\beta}\in (L^\theta X)_\beta\}$
such that
$\underline v'\leq\underline v$
(resp. $\underline\nu'\leq\underline\nu$)
if and only if $v'_\beta\leq v_\beta$
(resp. $\nu'_\beta\leq \nu_\beta)$ for all $\beta$.
For any $\underline v\in\underline\sV$ and $\underline\nu\in\underline{\sN}$, we set 
$\overline{\cO}_{\underline\nu}=\bigsqcup
\overline{\cO}_{\nu_\beta}$
and $\overline{L^\theta X}_{\underline\nu}=\bigsqcup
\overline{L^\theta X}_{\nu_\beta}$.

Combining Theorem \ref{Placidness of orbits} and Lemma \ref{filtered}, we obtain:

  \begin{prop}\label{equ dim placid}
      (i)  $L^\theta X\cong\on{colim}_{\underline\nu}\overline{L^\theta X_{\underline\nu}}$ is a locally equidimensional $L^+G$-placid presentation.

      (ii)  $L^\theta X\cong\on{colim}_{\underline v}\overline{\cO_{\underline v}}$ is a locally equidimensional $I_0$-placid presentation.
  \end{prop}

\subsection{Dimension theory for $L^\theta X$}
Following  \cite[Appendix A]{CYUntwisted},
a dimension theory for a placid ind-scheme  $Y$
is an assignment  
to  each locally equidimensional placid subscheme 
$S\subset Y$
a locally constant function
$\delta_{S}:S\to\mathbb Z$
such that for any finitely presented locally closed 
embedding $S\subset S'\subset Y$ we have 
\begin{equation}\label{dim_S/S'}
\delta_{S}-\delta_{S'}|_{S}=\on{dim}_{S/S'}:S\to\mathbb Z,
\end{equation}
here
$\dim_{S/S'}:S\to\mathbb Z$ 
is the dimension function associated to 
the inclusion $S\subset S'$.

\begin{lem}\label{dim}
    Assume $Y$ admits a locally equidimensional placid presentation 
    $Y\cong\on{colim}_{i\in I} Y^i$.
    Then there exists a dimension theory $\delta$ on $Y$.
\end{lem}
\begin{proof}
By \cite[A.4 (11)]{CYUntwisted},
it suffices to define $\delta_{Y^i}:Y^i\to\mathbb Z$ for each 
$Y^i$ satisfying~\eqref{dim_S/S'}.
For this, we fix an $i_0\in I$, 
and for $Y^i$ we choose a large enough orbit $i'\in I$ such that $i,i_0\leq i'$, and we define
\[
\delta_{Y^i}=\dim_{Y^{i}/Y^{i'}}-\dim_{Y^{i_0}/Y^{i'}}:Y^i\to\bZ.
\]
This is independent of the choice of $i'\in I$
and satisfies ~\eqref{dim_S/S'}, 
since if $j\in I$ such that $i_0,i,i'\leq j$, we have 
\begin{align*}
\dim_{Y^i/Y^{j}}-\dim_{Y^{i_0}/Y^j}
&=(\dim_{Y^i/Y^{i'}}+\dim_{Y^{i'}/Y^{j}})-(\dim_{Y^{i_0}/Y^{i'}}+\dim_{Y^{i'}/Y^j})\\
&=\dim_{Y^{i'}/Y^{i'}}-\dim_{Y^{i_0}/Y^i}
\end{align*}
and 
\[
\delta_{Y^i}-\delta_{Y^{i'}}|_{Y^i}
=(\dim_{Y^{i}/Y^{j}}-\dim_{Y^{i_0}/Y^j})-(\dim_{Y^{i'}/Y^{j}}-\dim_{Y^{i_0}/Y^{j}})
=\dim_{Y^{i}/Y^{i'}}.
\]
The proposition follows.
\end{proof}

\begin{prop}\label{p:dim theory}
    There is a dimension theory $\delta$
    on $L^\theta X$.
\end{prop}
\begin{proof}
    This follows from Proposition \ref{equ dim placid} and Lemma \ref{dim}.
\end{proof}

\begin{rem}
In \cite[Proposition 41]{CYUntwisted}, we deduce the existence of a dimension theory for the untwisted loop space $LX$ from the fact that closed Iwahori orbits lying in the same connected component of $LX$ have the same relative dimension; see \cite[Proposition 40]{CYUntwisted}. 
We do not know whether the analogous property holds for $L^\theta X$. 
Lemma~\ref{filtered} provides an alternative argument for the existence of a dimension theory that applies uniformly to both the untwisted and twisted settings.
\end{rem}

\section{Transversal slices}\label{Transversal slices}
We construct transversal slices of 
spherical and Iwahori obits on $L^\theta X$ 
and establish some basic properties of them.
The construction generalizes the one for the untwisted case
\cite[\S6]{CYUntwisted}.

\subsection{Slices for spherical orbits}
\subsubsection{Involution on affine Grassmannian slices}
Let $L^-G=G[t^{-1}]$, $L^{<0}G=\ker(L^-G\rightarrow G)$, and $L^-\fg,L^{<0}\fg$ their Lie algebras. 
For any dominant coweight $\lambda\in\Lambda_{T_0}^+$, 
we write $LG_\lambda=L^+Gt^\lambda L^+G$
for the $L^+G\times L^+G$-orbit of $t^\lambda$ in $LG$,
and $\overline {LG_\lambda}$ its closure.
The affine Grassmannian slice of $LG_\lambda$ in $LG$ is defined as 
$W^\lambda=(L^{<0}G\cap \Ad_{t^\lambda}L^{<0}G)t^\lambda$.
For any pair $\lambda,\mu\in\Lambda_{T_0}^+$, define
$W^\lambda_\mu=W^\lambda\cap LG_\mu$ and $W^\lambda_{\leq\mu}=W^\lambda\cap \overline{LG_\mu}$.
Some well-known facts on these slices that we need are 
listed in \cite[Lemma 42]{CYUntwisted}.

Recall from Proposition \ref{p:spherical orbits} that
$L^+G$-orbits on $L^\theta X$ 
are represented by standard spherical representatives
$x=t^\lambda g_x$, $g_x\in G$.
Let $y=t^\mu g_y$ be another such element.
Denote
\begin{align*}
&W^x:=(L^{<0}G\cap\Ad_xL^{<0}G)x=W^\lambda g_x,\\
&W^x_y:=W^x\cap L^+GyL^+G=W^\lambda g_x\cap LG_\mu=W^\lambda_\mu g_x,\\
&W^x_{\leq y}:=W^x\cap\overline{L^+GyL^+G}=W^\lambda_{\leq\mu}g_x.
\end{align*}

\begin{lem}\label{l:inv on spherical slices}
Let $x=t^\lambda g_x,y=t^\mu g_y\in(t^\lambda G)^{\inv\circ\theta}$.
The 
map $\on{inv}\circ\theta:LG\to LG$
restricts to involutions on $W^x$, $W^x_y$, and $W^x_{\leq y}$.
Moreover, the identification 
$W^x\cong L^{<0}G\cap \Ad_{t^\lambda}L^{<0}G$, sending $\gamma x\to \gamma$, intertwines the involution $\on{inv}\circ\theta|_{W^x}$
with the involution $\on{inv}\circ\psi_x$
on $L^{<0}G\cap \Ad_{t^\lambda}L^{<0}G$
where $\psi_x=\on{Ad}_x\circ\theta$.
\end{lem}
\begin{proof}
Since $x$ is fixed by $\on{inv}\circ\theta$,
$LG_\lambda$ is stable under $\on{inv}\circ\theta$
and $L^{<0}G\cap \Ad_{t^\lambda}L^{<0}G$ is stable under 
$\on{inv}\circ\psi_x$.
Thus for any 
$\gamma x\in W^x=(L^{<0}G\cap \Ad_{t^\lambda}L^{<0}G)x$
(resp. $W^x_y$, $W^x_{\leq y}$), we have 
\[
\on{inv}\circ\theta(\gamma x)=\theta(x)^{-1}\theta(\gamma)^{-1}=(\Ad_x\theta(\gamma)^{-1})x=\on{inv}\circ\psi_x(\gamma) x\in
 W^x\ \    (resp.\ \  W^x_y, W^x_{\leq y}).
\]
The lemma follows.
\end{proof}

Now write $x=t^\lambda g_x=t^\lambda g_0w_1^{-1}\in L^\theta X$,
$g_0\in L_\lambda$ in the Levi subgroup,
$\theta_0(\lambda)=-w_1\lambda$.
Consider the $\bGm$-action 
\begin{equation}\label{eq:spherical contraction}
	\rho_x(s):LG\to LG,
	\gamma(t)\to s^\lambda\gamma(s^{-2}t)s^{w_1\lambda}.
\end{equation}

\begin{lem}\label{G_m action on Aff slices}
The two maps $\rho_x(s)$
and $\on{inv}\circ\theta$ commute with each other.
Moreover,
the  map $\rho_x(s)$ restricts to a $\mathbb G_m$-action $\rho_s$ on $W^x,W^x_y, W^x_{\leq y}$.
The $\mathbb G_m$-action $\rho_s$ on $W^x$ and $W^x_{\leq y}$
is contracting 
with unique fixed point $x$.
\end{lem}
\begin{proof}
First, for any $\gamma(t)x\in W^x$ 
where $\gamma(t)\in L^{<0}G\cap \Ad_{t^\lambda}L^{<0}G$,
we have 
\[
\rho_x(s)(\gamma(t)x)
=s^\lambda\gamma(s^{-2}t)(s^{-2}t)^\lambda g_0w_1^{-1} s^{w_1\lambda}
=(\Ad_{s^\lambda}\gamma(s^{-2}t))x\in W^x.
\]
Thus $\rho_x(s)$ preserves $W^x$.

Since $\gamma(t)\in L^{<0}G\cap\on{Ad}_{t^\lambda}L^{<0}G$,
$\Ad_{s^\lambda}\gamma(s^{-2}t)$ goes to the unit $1\in LG$ 
as $s$ goes to $0$. 
It follows that the $\mathbb G_m$-action on $W^\lambda$
is contracting with unique fixed point $t^\lambda$.
Since $LG_\mu$ is obviously stable under 
$\rho_x(s)$, the $\mathbb G_m$-action 
preserves $W^x_y$ and $W^x_{\leq y}$.
\end{proof}

Consider the fixed points 
\begin{equation}\label{L}
L^x=(W^x)^{\on{inv}\circ\theta}
\cong (L^{<0}G\cap \Ad_{t^\lambda}L^{<0}G)^{\on{inv}\circ\psi_x},\quad
L^x_y=(W^x_y)^{\inv\circ\theta},\quad
L^x_{\leq y}=(W^x_{\leq y})^{\inv\circ\theta}.
\end{equation}

\begin{lem}\label{slice spherical}
\begin{itemize}
\item [(i)] 
$L^x$ is a connected  ind-scheme of ind-finite type and formally smooth.
\item [(ii)]
The map $\rho_x$ acts on $L^x$,
contracting it to the unique fixed point $x$.
\item [(iii)] The multiplication map $m:L^+G\times L^x\rightarrow L^\theta X$ is formally smooth.	
\item [(iv)] 
Denote $\cO_x=L^+G\cdot x$.
Then $L^\lambda\cap \cO_x=\{x_\lambda\}$ and 
$L^x\subset L^\theta X$ 
is transversal to $\cO_x$ at $x$. 
\item [(v)] 
All the $\bGm$-actions $\rho_x$ 
preserve all the $L^+G$-orbits $\cO_y$. 
Any $L^+G$-equivariant local system on  $\cO_y$ is $\bGm$-equivariant with respect to any $\rho_x$.
\end{itemize}
\end{lem}
\begin{proof}
(i) follows from \cite[Lemma 42, Lemma 97]{CYUntwisted} and
Lemma \ref{l:inv on spherical slices}.
(ii) follows from Lemma \ref{G_m action on Aff slices}.

(iii):
It sufficies to check the surjectivity of the differential map.
From standard results on affine Grassmanian slices and $W^x=W^\lambda g_x$, 
we know that the multiplication map
\[
m: L^+G\times W^x\times L^+G\rightarrow LG,\quad
(g_1,a,g_2)\mapsto g_1ag_2
\]
is a submersion.

Denote $\sigma=\on{inv}\circ\theta$, 
which is an involution of space on $LG$.
Also denote by $\gamma(g_1,a,g_2)=(\sigma(g_2),\sigma(a),\sigma(g_1)$,
which is an involution on $L^+G\times W^x\times L^+G$.
Its fixed subspace is 
\[
\{(g_1,a,\sigma(g_1))\mid g_1\in L^+G,a\in(W^x)^\sigma\}
\cong L^+G\times L^x.
\]

Observe that $m$ intertwines $\gamma$ and $\sigma$.
The induced map on fixed subspaces is isomorphic to 
$m:L^+G\times L^x\rightarrow (LG)^\sigma$.
Now for any $p=(g,a)\in L^+G\times L^x\hookrightarrow L^+G\times W^x\times L^+G$, the surjective differential map
\[
\td m: T_p(L^+G\times W^x\times L^+G)\rightarrow T_{m(p)}LG
\]
intertwines $\gamma$ and $\sigma$.
Since both $\gamma$ and $\sigma$ are involutions, 
the induced map on fixed subspaces,
which is $\td m:T_p(L^+G\times L^x)\rightarrow T_{m(p)}(LG)^\sigma$,
is also surjective as desired.

(iv):
$L^x\cap \cO_x=\{x\}$ follows from
$W^x\cap L^+GxL^+G=(W^\lambda\cap LG_{t^\lambda})g_x=\{t^\lambda g_x=x\}$.
To show $L^x$ is transversal to $\cO_x$ at $x$, we need to check
    \[
    T_x \cO_x\oplus T_x L^x=T_x(LG)^{-\psi_x}.
    \]
    Explicitly, 
    \[
    T_x\cO_x=\{X-\psi_x(X)|X\in L^+\fg\},\quad
    T_x L^x=(L^{<0}\fg\cap\Ad_xL^{<0}\fg)^{-\psi_x}.
    \]
    It follows from part (iii) that the LHS spans RHS.
    It only remains to show the LHS is a direct sum,
    which follows from the decomposition
    \[
    L\fg=L^+\fg\oplus(\Ad_xL^+\fg\cap L^{<0}\fg)\oplus
    (L^{<0}\fg\cap\Ad_xL^{<0}\fg).
    \]

(v): 
    Let $y=t^\mu g_y$ be a standard spherical representative.
	Denote $r_s(\gamma(t)):=\gamma(s^{-2}t)$.
    For $g\in G(\cO)$, we have
	\[
	\rho_x(s)(g\cdot y)
	=s^\lambda r_s(g)(s^{-2}t)^\mu g_y\theta(r_s(g))^{-1}s^{w_1\lambda}
	=(s^\lambda r_s(g)s^{-\mu})\cdot y.
	\]
	As $s^\lambda r_s(g)s^{-\mu}\in G(\cO)$, 
	$\rho_x$ preserves $\cO_y$.
	
	Now we show any local system $\cL$ on $\cO_y$ 
	is $\rho_x$-equivariant. 
	Consider $\bGm$-action on $L^+G$:
	$\gamma_s(g)=\Ad_{s^\lambda}r_s(g)$.
	The actions of $L^+G$ and $\bGm$  via $\rho_x$ on $\cO_y$ at $y$
	can be integrated into an action of the semidirect product
	$L^+G\rtimes\bGm$ with respect to $\gamma_s$.
	Consider exact sequence
	\[
	1\rightarrow
	\mathrm{Stab}_{L^+G}(y)\rightarrow \mathrm{Stab}_{L^+G\rtimes\bGm}(y)\rightarrow
	\bGm,\quad
	g\mapsto (g,1),\quad (g,s)\mapsto s.
	\]
	Consider section $p:\bGm\rightarrow\mathrm{Stab}_{L^+G\rtimes\bGm}(y), 
	p(s)=(s^{\mu-\lambda},s)$.
	This is clearly a group homomorphism. 
	Thus the above exact sequence is a split short exact sequence.
	We get semidirect product
	\[
	\mathrm{Stab}_{L^+G\rtimes\bGm}(y)\cong
	\mathrm{Stab}_{L^+G}(y)\rtimes\bGm.
	\]
	Thus the component groups are 
	\[\pi_0(\mathrm{Stab}_{L^+G\rtimes\bGm}(y))\cong
	\pi_0(\mathrm{Stab}_{L^+G}(y)\rtimes\bGm)\cong
	\pi_0(\mathrm{Stab}_{L^+G}(y)).
	\]
	Therefore $L^+G$-equivariant local systems on $\cO_y$ are $\rho_x$-equivariant for any $y$.
\end{proof}

Denote $S^x_y:=L^x\cap\cO_y$,
$S^x_{\leq y}=L^x\cap\overline{\cO}_y$.
We say $x\leq y$ if $x\in\overline{\cO}_y$.

\begin{cor}\label{c:S^x_y equi-singular}
	Let $x,y$ be standard spherical representatives.
	\begin{itemize}    
		\item [(i)] The multiplication map $m:L^+G\times S^x_{\leq y}\rightarrow\overline{\cO}_y$ is 
		formally smooth.
		\item [(ii)] The $\bGm$-action $\rho_x$ acts on $S^x_{\leq y}$, 
		which fixes $x$, and contracts $S^x_{\leq y}$ to $x$.
        \item [(iii)] $S^x_y\neq\emptyset$ if and only if $x\leq y$.
		
	\end{itemize}
\end{cor}
\begin{proof}
(i): This is the base-change of the formally smooth map
$L^+G\times L^x\to L^\theta X$ in Lemma \ref{slice spherical}.(iii) along the embedding $\overline{\cO}_y$, hence is formally smooth.

(ii): This follows from Lemma \ref{slice spherical}.(ii).

(iii):
Assume $S^x_y=L^x\cap\cO_y\neq\emptyset$. 
	Recall that we have the $\bGm$-action $\rho_x$ on $L^x$, 
	which also acts on all the $L^+G$-orbits $\cO_y$. 
	The contraction of $L^x\cap \cO_y$, 
	which gives $x$, 
	must be contained in $\overline{\cO}_y$. 
	Thus $x\leq y$.
	
	Conversely, if $x\leq y$, by definition 
	$x\in S^x_{\leq y}\neq\emptyset$.
    Since  $m:L^+G\times S^x_{\leq y}\to \overline{\cO}_y$ is formally smooth and $L^+G\times S^x_{\leq y}$ admits a placid 
    presentation with surjctive transition maps, 
    thus the image of $m$
    has nonempty intersection with the open dense subset $\cO_y$. 
    It follows that 
    $m^{-1}(\cO_y)=L^+G\times S^x_y\neq\emptyset$.
    Part (iii) follows.
\end{proof}

\subsection{Lagrangian and symplectic slices}
\subsubsection{Poisson structure}
We assume $p$ is large enough so that 
there is a non-degenerate symmetric bilinear invariant form 
$(\ ,\ )$ on $\fg$
that is also $\theta_0$-invariant.

Recall the affine Grassmannian slice $W^\lambda_{\leq\mu}$
has a natural Poisson structure and $W^\lambda_\nu$ for $\lambda\leq\nu\leq\mu$
are symplectic leaves of dimension 
$\dim W^\lambda_\nu=2\langle\rho,\nu-\lambda\rangle$
(see, e.g., \cite{KWWY}). 
Denote $\sigma=\inv\circ\theta$,
$L^\lambda_\mu=(W^\lambda_\mu)^\sigma$.
In the untwisted case we showed that $L^\lambda_\mu$ is a coisotropic subvariety of $W^\lambda_\mu$
\cite[Lemma 67]{CYUntwisted}.

We first review the construction of Poisson structure 
on the affine Grassmannian slice.
Recall the Manin triple $(L\fg,L^{<0}\fg,L^+\fg)$:
fix a non-degenerate symmetric bilinear invariant form $(\ ,\ )$ on $\fg$
that is also $\theta_0$-invariant.
Extend it $F$-linearly to $L\fg=\fg(F)$. 
Composing with residue map, we get a $k$-valued inner product on $L\fg$,
which we still denote by $(\ ,\ )$.
Then $L^{<0}\fg,L^+\fg$ are both Lagrangian subspaces of $L\fg$
with respect to $(\ ,\ )$.
Fix a basis $e_i^*$ of $L^{<0}\fg$
and a basis $e_i$ of $L^+\fg$
such that $(e_i^*,e_j)=\delta_{ij}$. 
Define $r$-matrix
\[
r=\sum_i e_i^*\wedge e_i\in\wedge^2L\fg,
\]
where the formal sum is understood as an element 
in the completion of $\fg(\!(t^{\pm})\!)\otimes\fg(\!(t^\pm)\!)$.
The definition of $r$ is independent of the choice of orthonormal basis. 
Since $\theta$ acts continuously on $L^{<0}\fg$ and $L^+\fg$, we have
\[
\theta(r)=r.
\]
Define a bi-vector field $\pi\in\Gamma(LG,\wedge^2TLG)$ by
\[
\pi(g):=R_{g*} r-L_{g*} r,
\]
where $g\in LG$, $L_g,R_g$ are left and right translations by $g$.
Then $\pi$ gives a Poisson bi-vector and defines a Poisson structure on $LG$. Moreover, $L^{<0}G,L^+G$ are Poisson subgroups of $LG$.
This induces a bi-vector field $\bar{\pi}$ on $\Gr$ from $r$.
Then $\bar{\pi}$ defines a Poisson structure on $\Gr$,
and both $L^+G$-orbits $p(LG_\lambda)=L^+G\cdot [t^\lambda]$
and $L^{<0}G$-orbits $p(W^\lambda)=L^{<0}G\cdot[t^\lambda]$
are Poisson subvarieties.
The intersections of $L^+G$-orbits and $L^{<0}G$-orbits, 
when nonempty, 
are symplectic leaves,
c.f. \cite[Corollary 2.9]{LYPoisson}, \cite[Theorem 2.5]{KWWY}.

\subsubsection{Untwisted Lagrangian slices}\label{sss:untwisted LS}
We first extend the result in the untwisted setting.
Only in this section, 
we use $\theta$ to denote the untwisted involution on $LG,L\fg$.
Recall we proved in \cite[Lemma 67]{CYUntwisted} that
$L^\lambda_\mu\subset W^\lambda_\mu$ is a coisotropic subvariety.
We now show that $L^\lambda_\mu\subset W^\lambda_\mu$ is also isotropic.
\begin{lem}\label{l:isotropic}
	\begin{itemize}
		\item [(i)] For $g\in LG$, $\pi(g^{-1})=-L_{g^{-1}*}R_{g^{-1}*}\pi(g)$.
		\item [(ii)] $\sigma_*(\pi)=-\pi$.
		\item [(iii)] $L^\lambda_\mu\subset W^\lambda_\mu$ is isotropic.
	\end{itemize}
\end{lem}
\begin{proof}
	(i). Since $L_{g*}$ and $R_{g*}$ commute, we have
	\[
	\pi(gh)=R_{gh*}r-L_{gh*}r
	       =R_{h*}R_{g*}r-L_{g*}L_{h*}r
	       =L_{g*}\pi(h)+R_{h*}\pi(g).
	\]
	Note that at unit, $\pi(e)=r-r=0$.
	Thus $0=\pi(e)=L_{g*}\pi(g^{-1})+R_{g^{-1}*}\pi(g)$,
	$\pi(g^{-1})=-L_{g^{-1}*}R_{g^{-1}*}\pi(g)$.
	
	(ii). For any $g\in LG$, $X\in T_gLG$, since $\inv(\exp(tX)g)=g^{-1}(\exp(-tX)g)g^{-1}$, we have
	\[
	\inv_{g*}(X)=-L_{g^{-1}*}R_{g^{-1}*}X.
	\]
	Thus for bivector $\pi(g)$, we get
	\[
	\inv_{g*}\pi(g)=(-1)^2L_{g^{-1}*}R_{g^{-1}*}\pi(g)
	=L_{g^{-1}*}R_{g^{-1}*}\pi(g).
	\]
	Combining this with (i), we obtain
	\[
	\pi(g^{-1})=-\inv_{g*}\pi(g).
	\]
	Since $r$ is $\theta$-invariant, we conclude
	\[
	\sigma_*\pi(g)
	=\inv_{\theta(g)*}(R_{\theta(g)*}\theta(r)-L_{\theta(g)*}\theta(r))
	=\inv_{\theta(g)*}\pi(\theta(g))
	=-\pi(\theta(g)^{-1})
	=-\pi(\sigma(g)).
	\]
	
	(iii). Let $\omega$ be the symplectic form on $W^\lambda_\mu$ induced from $\bar{\pi}$.
	Explicitly, for $g\in W^\lambda_\mu$, $v,w\in T_gW^\lambda_\mu$,
	let $v^*,w^*\in T_g^*W^\lambda_\mu$ be such that
	$v=\langle\bar{\pi}(g),v^*\rangle,w=\langle\bar{\pi}(g),w^*\rangle$.
	Then $\omega_g(v,w)=\bar{\pi}(g)(v^*,w^*)$, which can be verified to be well-defined.
	
	By (ii) and two definitions of $\bar{\pi}$ \cite[Lemma 66]{CYUntwisted}, 
    we also have $\sigma_*\bar{\pi}=-\bar{\pi}$.
	Thus we have
	\[
	(\sigma^*\omega)_g(v,w)=\sigma_*\bar{\pi}(g)(\sigma_g^*v^*,\sigma_g^*w^*)
	=-\bar{\pi}(\sigma(g))(\sigma_g^*v^*,\sigma_g^*w^*)
	=-\omega_{\sigma(g)}(\sigma_{g*}v,\sigma_{g*}w).
	\]
	
	Let $g\in L^\lambda_\mu=(W^\lambda_\mu)^\sigma$
	and $v,w\in T_gL^\lambda_\mu$. 
	Note $\sigma_{g*}v=v,\sigma_{g*}w=w$.
	We obtain
	\[
	\omega_g(v,w)=\omega_{\sigma(g)}(\sigma_{g*}v,\sigma_{g*}w)
	=(\sigma^*\omega)_g(v,w)
	=-\omega_g(v,w).
	\]
	Therefore $\omega_g(v,w)=0$,
	i.e. $L^\lambda_\mu$ is isotropic.
\end{proof}

Combining \cite[Lemma 67]{CYUntwisted} and Lemma \ref{l:isotropic}, we get
\begin{cor}\label{c:Lagrangian}
	For $\lambda,\mu\in\Lambda_S^+$, 
	$L^\lambda_\mu$ is Lagrangian in $W^\lambda_\mu$.
	In particular, 
	$\dim L^\lambda_\mu=\frac{1}{2}\dim W^\lambda_\mu=\langle\rho,\mu-\lambda\rangle$.
\end{cor}
This generalizes the special case 
\cite[Proposition 68]{CYUntwisted} of $\lambda=0$.

\subsubsection{Twisted symplectic slices}\label{sss:twisted SS}
Henceforth we switch back to the twisted involution $\theta$
on $LG,L\fg$.
Let $x,y$ be standard spherical representatives.
Recall slice $W^x=W^\lambda g_x$ and $W^x_y=W^\lambda_\mu g_x$.
We define the Poisson structure on $W^x$ and 
symplectic structure on $W^x_y$
via the isomorphism to $W^\lambda,W^\lambda_\mu$ 
by right multiplication by $g_x^{-1}$.

Denote $\sigma=\inv\circ\theta$,
which acts on $W^x,W^x_y$.
It induces an involution 
$\sigma_x(a)=\theta(g_x)^{-1}\theta(a)^{-1}g_x^{-1}$
on $W^\lambda,W^\lambda_\mu$ via the above isomorphism.
Denote
\[
{}_xL^\lambda=(W^\lambda_\mu)^{\sigma_x}=L^xg_x^{-1},\quad
{}_xL^\lambda_\mu=(W^\lambda_\mu)^{\sigma_x}=L^x_y g_x^{-1}.
\]

Thus equivalently, the subjects under study are subvarieties
${}_xL^\lambda,{}_xL^\lambda_\mu$ of $W^\lambda,W^\lambda_\mu$. 
As in the untwisted case \cite[Lemma 45 (ii)]{CYUntwisted}, 
${}_xL^\lambda_\mu$ is smooth. 
We impose the same Poisson structure on $W^\lambda$ as in the untwisted case
by the same $r$-matrix $r=\sum_i e_i^*\wedge e_i$
where $e_i^*,e_i$ are orthonormal basis of $L^{<0}\fg,L^+\fg$.
Recall $\fg=\fk\oplus\fp$,
so that the $\theta_0$-invariant $(\ ,\ )$ is non-degenerate on
$\fk\times\fk,\fp\times\fp$.
We take orthonormal basis $e_\alpha$ of $\fk,\fp$
and let $e^*_i=t^{-j-1}e_\alpha,e_i=t^je_\alpha$, $j\geq0$.

\begin{lem}\label{l:twisted Poisson}
	\begin{itemize}
		\item [(i)]
		$\theta(r)=-r$.
		\item [(ii)]
		The induced form $(\ ,\ )$ on $L\fg$ is $\theta$-anti-invariant.
		\item [(iii)]
		$(L\fg)^\theta=k(\!(t^2)\!)\fk\oplus tk(\!(t^2)\fp$,
		$(L\fg)^{-\theta}=tk(\!(t^2)\!)\fk\oplus k(\!(t^2)\!)\fp$.
		The orthogonal complements of 
		$(L\fg)^\theta,(L\fg)^{-\theta}$ are themselves.
		\item [(iv)]
		For $g\in G$, $\Ad_g(r)=r$.
		\item [(v)] 
		$\sigma_{x*}(\pi)=\pi$.
		\item [(vi)]
		Let $\omega$ be the symplectic form on $W^\lambda_\mu$ induced from $\bar{\pi}$.
		Then for any $g\in W^\lambda_\mu$ and $v,w\in T_gW^\lambda_\mu$,
		$(\sigma_x^*\omega)_g(v,w)=\omega_{\sigma_x(g)}(\sigma_{xg*}v,\sigma_{xg*}w)$.
	\end{itemize}
\end{lem}
\begin{proof}
(i).
$\theta(r)=\sum_i(-1)^{-j-1+j}e_i^*\wedge e_i=-r$.

(ii).
$(\theta(t^ie),\theta(t^jf))=(-1)^{i+j}\delta_{i+j,-1}(e,f)=-(t^ie,t^jf)$.

(iii). Follows by straightforward computation.

(iv).
Since $(\ ,\ )$ is adjoint invariant, $\Ad_g\in\mathrm{GL}(\fg)$
is orthogonal with respect to the inner product,
so that $\{\Ad_ge_\alpha\}$ is another set of orthonormal basis.
Then it is elementary to see 
$\sum_\alpha t^{-j-1}\Ad_g e_\alpha\wedge t^j\Ad_g e_\alpha=\sum_\alpha t^{-j-1}e_\alpha\wedge t^j e_\alpha$ for any $j$.	

(v).
For any $g\in LG$, $X\in T_gLG$, as in Lemma \ref{l:isotropic}, we have
$\pi(g^{-1})=-\inv_{g*}\pi(g)$.
By $\theta(r)=-r$, we have
\begin{align*}	
	\sigma_{x*}\pi(g)
	&=R_{g_x^{-1}*}L_{\theta(g_x)^{-1}*}\inv_{\theta(g)*}(R_{\theta(g)*}\theta(r)-L_{\theta(g)*}\theta(r))\\
	&=-R_{g_x^{-1}*}L_{\theta(g_x)^{-1}*}\inv_{\theta(g)*}\pi(\theta(g))\\
	&=R_{g_x^{-1}*}L_{\theta(g_x)^{-1}*}\pi(\theta(g)^{-1})\\
	&=\pi(\sigma_x(g)).
\end{align*}
In the last equality we used (iv).

(vi).
Explicitly, for $g\in W^\lambda_\mu$, $v,w\in T_gW^\lambda_\mu$,
let $v^*,w^*\in T_g^*W^\lambda_\mu$ such that
$v=\langle\bar{\pi}(g),v^*\rangle,w=\langle\bar{\pi}(g),w^*\rangle$.
Then $\omega_g(v,w)=\bar{\pi}(g)(v^*,w^*)$.
We also have $\sigma_{x*}\bar{\pi}=\bar{\pi}$.
Thus we have
\[
(\sigma_x^*\omega)_g(v,w)=\sigma_{x*}\bar{\pi}(g)(\sigma_{xg}^*v^*,\sigma_{xg}^*w^*)
=\bar{\pi}(\sigma_x(g))(\sigma_{xg}^*v^*,\sigma_{xg}^*w^*)
=\omega_{\sigma_x(g)}(\sigma_{xg*}v,\sigma_{xg*}w).
\]
\end{proof}

\begin{prop}\label{p:symplectic slice}
	\begin{itemize}
		\item [(i)]
		${}_xL^\lambda_\mu\subset W^\lambda_\mu$ is a symplectic subvariety.
		In particular, any component of ${}_xL^\lambda_\mu$ is even dimensional.
		\item [(ii)]
		If ${}_xL^\lambda_\mu\subset W^\lambda_\mu$ is positive dimensional,
		it is neither isotropic nor coisotropic.
	\end{itemize}
\end{prop}
\begin{proof}
	(i).
	Take $a\in {}_xL^\lambda_\mu$, a $\sigma_x$-fixed point.
	Denote $V=T_aW^\lambda_\mu$, 
	on which the differential of the involution $\sigma_x$ acts.
	Denote $V=V^+\oplus V^-$ the $+1,-1$ eigenspaces.
	Then $T_a({}_xL^\lambda_\mu)=V^+$.
	By Lemma \ref{l:twisted Poisson}.(vi), 
	the symplectic form $\omega_a$ on $T_aW^\lambda_\mu$ is $\sigma_x$-invariant.
	Thus for any $v\in V^+,w\in V^-$,
	\[
	\omega_a(v,w)=\omega_a(\sigma_x(v),\sigma_x(w))=-\omega_a(v,w),\quad
	\omega_a(V^+,V^-)=0.
	\]
	
	We need to show $\omega_a|_{V^+}$ is non-degenerate.
	For any $v\in V^+$, we need to find $w\in V^+$ such that 
	$\omega_a(v,w)\neq0$.
	Since $\omega_a$ is non-degenerate on $V$, there exists
	$w=w_++w_-\in V$, $w_+\in V^+,w_-\in V^-$, such that
	$0\neq\omega_a(v,w)=\omega_a(v,w_+)$, 
	which completes the proof.
	Thus $(V^+,\omega_a|_{V^+})$ is a symplectic space,
	which is even dimensional.
	
	(ii).
	Keep the notations as in part (i).
	Denote by $(V^+)^\perp$ the orthogonal complement of $V^+\subset V$
	with respect to $\omega_a$.
	By (i), we see $V^+\cap(V^+)^\perp=0$.
	Thus $V^+,(V^+)^\perp$ are not contained in each other if nonzero.
\end{proof}

\begin{rem}
	We explain why the proof for isotropicity and coisotropicity 
	for the untwisted involution fails for the twisted involution.
	In the untwisted case, the proof for both properties boils down to
	the vanishing of a number $c$, which is proved by $c=-c$.
	
	For isotropicity, $c=\omega(v,w)$, and the minus sign comes from
	$\sigma^*\omega=-\omega$.
	However, in the twisted case $\sigma^*\omega=\omega$ as in Lemma \ref{l:twisted Poisson}.(vi), so that there is no vanishing argument.
	
	For coisotropicity, $c=(Ad_{h^{-1}}r-\Ad_{\theta(h)^{-1}}r,z\wedge w)$,
	where in the twisted case $\theta_0(r)=r,\theta_0(z)=-z,\theta_0(w)=-w$,
	$(\ ,\ )$ is $\theta_0$-invariant, so that $c=\theta(c)=-c$.
	But in the twisted case, $\theta(r)=-r$, $\theta(z)=-z,\theta(w)=-w$.
	Also, $(\ ,\ )$ is $\theta$-anti-invariant on vectors, 
	thus $\theta$-invariant on bivectos.
	Therefore $c=\theta(c)=c$, and there is no vanishing argument.
\end{rem}

\subsubsection{Orbit slices}
\begin{lem}\label{l:orbit slice}
	Let $x\leq y$ be standard spherical representatives.
	\begin{itemize}
		\item [(i)]
		$\cO_y$ is open in $(L^+GyL^+G)^\sigma$.
		\item [(ii)]
		$S^x_y=L^x\cap\cO_y$ is open in $L^x_y=L^x\cap(L^+GyL^+G)^\sigma$.
		As a result, $S^x_y$ is even-dimensional. 
	\end{itemize}
\end{lem}
\begin{proof}
	(i).
	We show the inclusion $\cO_y\subset(L^+GyL^+G)^\sigma$
	is a submersion.
	Denote $\sigma_y=\Ad_y\circ\sigma$,
    with differential $d\sigma_y$.
	Since $L^+GyL^+G=L^+G\Ad_y L^+Gy$,
	we have
	\[
	T_y(L^+GyL^+G)^\sigma=R_{y*}(L^+\fg+\Ad_yL^+\fg)^{\td\sigma_y}.
	\]
	On the other hand,
	\[
	T_y\cO_y=R_{y*}\{X+\td\sigma_y(X)\mid X\in L^+\fg\}.
	\]
	Thus for any $X+\Ad_y Y=\td\sigma_y(X)+\td\sigma(Y)\in(L^+\fg+\Ad_y L^+\fg)^{\td\sigma_y}$,
	$X,Y\in L^+\fg$,
	let $Z=\frac{1}{2}(X+\td\sigma(Y))\in L^+\fg$.
	We have
	\[
	Z+\td\sigma_y(Z)
	=\frac{1}{2}(X+\td\sigma(Y)+\td\sigma_y(X)+\Ad_y(Y))
	=X+\Ad_yY.
	\]
	This completes the proof.
	
	(ii).
	This immediately follows from part (i) and Proposition \ref{p:symplectic slice}.(i).
\end{proof}

\subsection{Orthogonal and symplectic nilpotent orbits}
We discuss slice realization of some nilpotent orbits 
parallel to \cite[\S6.2]{CYUntwisted}.

Let $G=\GL_n$ and $\theta_0=(g^t)^{-1}$
(resp. $\theta_0(g)=J^{-1}(g^t)^{-1}J$ for $n=2m$ even, $J=\begin{pmatrix}0&I_m\\ -I_m&0\end{pmatrix}$).
Then $K=G^{\theta_0}=\on{O}_n$ (resp. $K=\Sp_n$).
We have 
$\fg=\mathfrak{gl}_{n}$, 
$\td\theta_0(g)=-g^t$ (resp. $\td\theta(g)=-J^{-1}g^tJ$), 
$\fk=\fg^{\td\theta_0}=\mathfrak{o}_n$ (resp. $\fk=\fsp_{n}$).
Denote by $\sN\subset\mathfrak g$ the nilpotent cone.
Denote the nilpotent orbit associated to 
a partition $\lambda\in\cP_n$ of $n$ by $\cO_\lambda$. 
Denote $\sN_\fk=\sN\cap\fk$.
The $K$-orbits of $\sN_\fk$ are classified by
even (resp. odd) partitions $\cP_{n,\fk}$, 
i.e. partitions whose even (resp. odd) parts occur with even multiplicity.
For $\lambda\in\cP_{n,\fk}$,
denote the associated $K$-orbit by $\cO_{\fk,\lambda}$.

Identify the set $\Lambda_T^+$ with 
$\Lambda_T^+=\{(\lambda_1\geq \lambda_2\geq\cdot\cdot\cdot\geq \lambda_n)\}$
so that 
\[LG_\lambda=L^+G\on{diag}(t^{\lambda_1-1},...,t^{\lambda_n-1})L^+G\]
for $\lambda\in\Lambda_T^+$.
Here we abuse the notation by using the translation $\lambda-1$,
which is used only in this section.
Recall the Lusztig embedding 
\begin{equation}\label{eq:GLn Lusztig embed iota_n}
	\iota:\sN\hookrightarrow W^0=L^{<0}G,\quad x\mapsto(I_n-t^{-1}x).
\end{equation}
It induces an isomorphism 
\[\sN\cong W^0_{\leq\omega_n}=L^{<0}G\cap \overline{LG}_{\omega_n}\]
where $\omega_n=(n\geq0\geq\cdots\geq0)$. 
In addition, for $\lambda\in\cP_n$, the above isomorphism 
restricts to isomorphisms
\[\cO_\lambda\cong W^0_\lambda,\ \ \ \ \ \ \ 
\overline{\cO}_\lambda\cong W^0_{\leq\lambda}.\]
 
Note that the Lusztig embedding intertwines the 
involution 
$\td\theta_0$ on $\sN$ with $\on{inv}\circ\theta$ on $W^0$
and hence induces isomorphisms
\[
\sN_\fk\cong (W^0_{\leq\omega_n})^{\on{inv}\circ\theta},
\quad
\cO_{\fk,\lambda}\cong(W^0_\lambda)^{\inv\circ\theta},
\quad
\overline\cO_{\fk,\lambda}\cong (W^0_{\leq\lambda})^{\on{inv}\circ\theta},
\quad
\lambda\in\cP_{n,\fk}
\]
on the fixed points that is equivariant under conjugation by $K$.
For each $\lambda\in\cP_{n,\fk}$,
fix a representative $c_\lambda\in\cO_{\fk,\lambda}$
and let $x_\lambda=I_n-t^{-1}c_\lambda\in(W^0_\lambda)^{\inv\circ\theta}$.
Note that $x_\lambda$ is not a spherical standard representative
in general.
Write $x_\lambda=g_\lambda t^\lambda y_\lambda\theta(g_\lambda)^{-1}$
where $g_\lambda\in G(\cO)$, $t^\lambda y_\lambda$ is a standard spherical representative.
Applying the twisted conjugation by $g_\lambda$ to 
Corollary \ref{c:S^x_y equi-singular},
we obtain that
$L^{x_\lambda}:=g_\lambda L^{t^\lambda y_\lambda}\theta(g_\lambda)^{-1}$
is transversal to $\cO_{x_\lambda}$ at $x_\lambda$

\begin{prop}\label{slice=nilp}\mbox{}
For any $\lambda\in\cP_{n,\fk}$,
we have $\mathbb G_m$-equivariant isomorphisms
\[
\cO_{\fk,\lambda}
\cong(W^0_\lambda)^{\inv\circ\theta}
=S^0_{x_\lambda},
\qquad
\overline\cO_{\fk,\lambda}\cong S^0_{\leq x_\lambda}.
\]
Moreover, for any $\lambda,\mu\in\cP_{n,\fk},\ \lambda\leq\mu$, 
the	variety $L^{x_\lambda}\cap S^0_{\leq x_\mu}$ defines a $\bGm$-contracting transversal slice 
to the nilpotent orbit
$\cO_{\fk,\lambda}$ inside $\overline\cO_{\fk,\mu}$.	
\end{prop}
\begin{proof}
since $\cO_{\fk,\lambda}$ is a single $K$-orbit,
$(W^0_\lambda)^{\inv\circ\theta}$
is a single $\theta$-twisted $K$-orbit
containing $S^0_{x_\lambda}$.
Since $S^0_{x_\lambda}$ is stable under the $K$-action,
the first isomorphism follows.

To prove the second isomorphism,
we have inclusion
$\overline\cO_{\fk,\lambda}\cong (W^0_{\leq\lambda})^{\on{inv}\circ\theta}\supset S^0_{\leq x_\lambda}$.
Conversely,
$S^0_{\leq x_\lambda}=L^0\cap\overline{\cO}_{x_\lambda}$
is closed in $L^0\cap\overline{LG}_\lambda^{\inv\circ\theta}=(W^0_{\leq\lambda})^{\on{inv}\circ\theta}\cong\overline\cO_{\fk,\lambda}$.
Since $\cO_{\fk,\lambda}\subset S^0_{\leq x_\lambda}$
is open dense in $\overline\cO_{\fk,\lambda}$,
we obtain $\overline\cO_{\fk,\lambda}\cong S^0_{\leq x_\lambda}$.

The $\bGm$-equivariance of the isomorphisms 
follow from Lemma \ref{slice spherical}.

The properties of the slice 
$L^{x_\lambda}\cap S^0_{\leq x_\mu}$
follow from a base change of 
Corollary \ref{c:S^x_y equi-singular} to $S^0_{\leq x_\mu}$.
\end{proof}
	
\begin{rem}\label{LWW}
For $G$ of simply-laced type and $\theta_0$ corresponding to a quasi-split Satake diagram, 
the fixed locus $(W^\lambda_\mu)^{\inv\circ\theta}$ 
appeared in the work \cite{LWWislice}, where it is called an \emph{affine Grassmannian islices}.
In \cite[Theorem 6.6]{LWWislice},
   it is proved that under Mirkovi\'c-Vybornov isomorphism, 
   $(W^\lambda_{\leq\mu})^{\inv\circ\theta}$
   is isomorphic to the Slodowy slice to $\cO_{\fk,\lambda}$
   inside $\overline\cO_{\fk,\mu}$.   
  \quash{
   It is unclear whether our slice 
   $S^{x_\lambda}_{\leq x_\mu}\cap S^0_{\leq x_\mu}$
   is the same as this Slodowy slice
   under Mirkovi\'c-Vybornov isomorphism.
   Note that $(W^\lambda_{\leq\mu})^{\inv\circ\theta}\subset L^{<0}G t^\lambda$
   while $S^{x_\lambda}_{\leq x_\mu}\cap S^0_{\leq x_\mu}\subset L^{<0}G$,
   so they are disjoint as subspaces of $LG$. }
\end{rem}

\subsection{Slices for Iwahori orbits}

\subsubsection{Involutions on affine flag slices}
Let $T_0\subset B_0\subset G$ and $I_0$ be
$\theta_0$-stable 
maximal torus, Borel subgroup, and Iwahori of $LG$ as before.
Denote the opposite Iwahori by $I_0^-:=B_0^-L^{<0}G$, 
where $B_0^-$ is the Borel opposite to $B_0$.
Let $I_0^{--}=U_0^-L^{<0}G$, 
$U_0^-$ the unipotent radical of $B_0^-$.
For any  $w\in\widetilde W=N_{LG}(LT_0)/L^+T_0$ 
 we denote by $\Fl_w\subset\Fl=LG/I_0$ the corresponding $I_0$-orbit on the affine flag variety 
 and $\overline{\Fl}_w$ its closure.
 We denote by 
$LG_w,\overline{LG}_w\subset LG$ the pre-images of $\Fl,\overline{\Fl}_w$ in $LG$.

Let $n\in N_{LG}(LT_0)$ be a lifting of $w$.
The affine flag slice of $LG_w$ in $LG$ at $n$ is defined as 
\[
W^{n}=(I_0^{--}\cap\on{Ad}_{n} I_0^{--})n\subset LG.
\]
For any pair $n,n'\in N_{LG}(LT_0)$ with images $w,w'\in\widetilde W$, 
we define
$W^{n}_{n'}=W^n\cap LG_{w'}$ and $W^n_{\leq n'}=W^n\cap \overline{LG}_{w'}$.
Some well-known facts we need about these slices are
listed in \cite[Lemma 53]{CYUntwisted}.

Since $\theta(T_0)=T_0$,
the normalizer $N_{LG}(LT_0)$
is stable under the involution $\on{inv}\circ\theta$ of $LG$.
By the same proof as \cite[Lemma 54]{CYUntwisted}, we have
\begin{lem}\label{inv of W^n}
Let
$n,n'\in (N_{LG}(LT_0))^{\on{inv}\circ\theta}$.
Then $W^n$, $W^n_{n'}$, and $W^n_{\leq n'}$
are stable under 
the involution $\on{inv}\circ\theta:LG\to LG$.
Moreover, the identification 
$W^n\cong I_0^{--}\cap \Ad_{n} I_0^{--}$, sending $\gamma\cdot n\to \gamma $, intertwines the involution $\on{inv}\circ\theta|_{W^n}$
with the involution $\on{inv}\circ\psi_n$
on $I_0^{--}\cap \Ad_{n} I_0^{--}$,
where 
$\psi_n=\on{Ad}_n\circ\theta:LG\to LG$.
\end{lem}

Assume further that $n_v=n\in N_{LG}(LT_0)^{\on{inv}\circ\theta}$
has the form 
$n=\bar nt^\lambda$
where $\bar{n}\in N_{G}(T_0)$ and $\lambda\in X_*(T_0)$
	satisfy $\theta_0(\bar{n})=(-1)^{\theta_0(\lambda)}\bar{n}^{-1},\
	\theta_0(\lambda)=-\Ad_{\bar{n}}\lambda$.
Observe $\psi_n|_{T_0}=\psi_{\bar n}=\on{Ad}_{\bar n}\circ\theta_0:T_0\to T_0$.
Choose a cocharacter $\mu\in X_*(T_0^{\psi_{\bar{n}}})$ 
as in \cite[\S6.4]{MS}, so that
$\langle\mu,\alpha\rangle<0$ 
for all positive roots $\alpha>0$ with $\psi_{\bar{n}}\alpha>0$.
Let $m=\max\{1,\langle\mu,\alpha\rangle+1,\forall\alpha\in\Phi\}$ 
where $\Phi$ is the set of roots of $G$. 
Consider the following map on $LG$:
\begin{equation}
	\phi_s=\phi_s^v: 
	\gamma(t)\mapsto s^{\mu-m\theta_0(\lambda)}\gamma(s^{-2m}t)s^{-\theta_0(\mu)+m\lambda}. 
\end{equation}

By the same proof as \cite[Lemma 55]{CYUntwisted}, we have
\begin{lem}\label{G_m action on AF slices}
Let $n,n'\in N_{LG}(LT_0)^{\on{inv}\circ\theta}$
and assume $n=\bar n t^\lambda$ as above.
\begin{itemize}
\item [(i)]
The two maps 
$\phi_s$ and $\on{inv}\circ\theta$ commute with each other.
\item [(ii)]
$\phi_s$
restricts to a $\mathbb G_m$-action $\rho_s$ on $W^n,W^{n}_{n'}, W^{n}_{\leq n'}$.
The $\mathbb G_m$-action $\rho_s$ on $W^n$ and $W^n_{\leq n'}$
is contrating 
with unique fixed point $n$.
\end{itemize}
\end{lem}

We preserve the set ups in Lemma \ref{G_m action on AF slices}.
Consider the fixed points 
\begin{equation}\label{L}
L^n:=(W^n)^{\on{inv}\circ\theta}
\cong (I_0^{--}\cap \Ad_{n}I_0^{--})^{\on{inv}\circ\psi_n}.
\end{equation}

The following proposition implies  that 
$L^n$ defines a
transversal contracting slice 
to the orbit $\cO_v$ at $n_v$.
\begin{prop}\label{Iwahori Slice}
	\begin{itemize}
		\item [(i)] $L^n $ is a connected  ind-scheme of ind-finite type and formally smooth.
		\item [(ii)]
		The  $\mathbb G_m$-action $\rho_s$ on $W^n$ restricts to a
		$\mathbb G_m$-action 
		on $L^n$, contrating it to the unique fixed point $n$.
		
		\item [(iii)] The action $m:I_0\times L^n\rightarrow L^\theta X$ 
			is formally smooth.	
		\item [(iv)] 
		Let $n=n_v$.
		We have $L^n\cap \cO_v=\{n\}$
		and 
		$L^n\subset L^\theta X$ is transversal to $\cO_v$ at $n$. 
		
		\item [(v)] All the $\bGm$-actions $\rho_s^v$ 
		preserve all the $I_0$-orbits $\cO_u$. 
		Any $I_0$-equivariant local system on  $\cO_u$ is $\bGm$-equivariant with respect to the action $\rho_s^v$. 
	\end{itemize}
\end{prop}

\begin{proof}
(i)-(ii): Using Lemma \ref{G_m action on AF slices}, 
the proof is the same as \cite[Lemma 56 (i),(iv)]{CYUntwisted}.

(iii):
Similarly as Lemma \ref{slice spherical}.(iii), 
replacing $L^+G$ with $I_0$. 
More precisely, it is well-known the multiplication map
\[
m:I_0\times W^n\times I_0\rightarrow LG
\]
is formally smooth.
Take $\sigma=(\on{inv}\circ\theta)$-fixed subspaces.
Since $m$ intertwines involutions
$\sigma(g_1,x,g_2)=(\sigma(g_2),\sigma(x),\sigma(g_1))$ on $I_0\times W^n\times I_0$
and $\sigma$ on $LG$.
Thus it induces a surjective differential map on fixed subspaces.

(iv):
$L^n\cap \cO_v=\{n\}$ follows from
$W^n\cap I_0nI_0=\{n\}$.
To show $L^n$ is transversal to $\cO_v$ at $n$, we need to check
\[
T_{n}\cO_v\oplus T_{n}L^n=T_{n}(LG)^{-\psi_n}.
\]
Explicitly, 
\[
T_{n}\cO_v=\{X-\psi_n(X)|X\in \Lie(I_0)\},\quad
T_{n}L^n=(\Lie(I_0^{--})\cap\Ad_{n}\on{Lie}(I_0^{--}))^{-\psi_n}.
\]
It follows from part (v) that the LHS spans RHS.
It only remains to show the LHS is a direct sum,
which follows from the decomposition
\[
L\fg=\on{Lie}(I_0)\oplus(\Ad_n\on{Lie}(I_0)\cap \on{Lie}(I_0^{--}))\oplus
(\on{Lie}(I_0^{--})\cap\Ad_n\on{Lie}(I_0^{--})).
\]
\\

(v): 
Similarly as Lemma \ref{slice spherical}.(iii), 
replacing $L^+G$ with $I_0$. 
We only verify $\rho_s^v$ preserves $\cO_u$.
Assume $n_u=\bar{n'}t^\alpha$, $\Ad_{\bar{n'}}\alpha=-\theta_0(\alpha)$.
Then for $h\in I_0$, 
\[
\rho^v_s(h\cdot_\theta n_u)
=(s^{\mu-m\theta_0(\lambda)}hs^{m\theta_0(\alpha)})\cdot_\theta n_u.
\]
\end{proof}

Denote $S^{n_v}_{n_u}=L^{n_v}\cap\cO_u$,
$S^{n_v}_{\leq u}=L^{n_v}\cap\overline\cO_u$.
By the same proof as \cite[Corollary 59]{CYUntwisted}, we obtain
\begin{cor}\label{c:bar O equi-singular}
Let $v,u\in\sV$. 
	\begin{itemize}
		\item [(i)] 
            The multiplication map 
            $m:I_0\times S^{n_v}_{\leq n_u}\rightarrow\overline\cO_u$ 
            is formally smooth.
        
		\item [(ii)] 
            The $\bGm$-action $\rho_s^v$ acts on $S^{n_v}_{\leq n_u}$, 
            which fixes $n_v$, and contracts $S^{n_v}_{\leq n_u}$ to $n_v$.

            \item [(iii)]
            $S^{n_v}_{n_u}\neq\emptyset$ if and only if $v\leq u$.
	\end{itemize}
\end{cor}

\section{Twisted Affine Hecke modules}\label{Affine Hecke modules}
\subsection{Twisted Affine Lusztig-Vogan modules}
\subsubsection{}
We assume $G$, $B_0$, $T_0$, $\theta$ are defined over 
a finite field $\mathbb F_q$ and $T_0$ is split 
over $\mathbb F_q$. 
Let $\sD$ be the set of pairs 
$(\xi,v)$ where $v\in\sV$ and 
$\xi$ is an $I_0$-equivariant local system on 
the  $I_0$-orbit $\cO_v\subset L^\theta X$.
Let $M$ be the free $\mathbb Z[q^{\frac{1}{2}},q^{\frac{-1}{2}}]$-module with basis 
indexed by $\sD$.
In this section we endow 
$M$ with a module structure over the affine Hecke algebra of $G$
and an anti-linear involution $D_\delta:M\to M$ 
compatible with the module structure.
In view of Theorem \ref{Placidness of orbits} and
Proposition \ref{p:dim theory},
most arguments are the same as 
the untwisted case \cite[\S8]{CYUntwisted}.
We will focus on formulating the statements 
and making necessary adjustment.

\subsubsection{}
Denote by $\pi_{T_0}:I_0\rightarrow I_0/I_0^+\cong T_0$ the quotient map.
For an orbit $\cO_v$,
note that $T_0$ is stable under the involution $\psi_v=\Ad_{n_v}\circ\theta$.
Thus $\pi_{T_0}(I_0\cap(LG)^{\psi_v})=T_0^{\psi_v}$.
Denote $T_v:=T_0^{\psi_v}$.
The $T_0$-equivariant rank one local systems on $T_0/T_v$ are given by
characters $X(\pi_0(T_v))$.
For $\xi\in X(\pi_0(T_v))$,
denote the corresponding local system on $T_0/T_v$ by $\cL_\xi$.

As in the untwisted case \cite[Lemma 71]{CYUntwisted},
the $I_0$-equivariant local systems on an orbit $\cO_v$ are
direct sums of rank one local systems of the form
$\cL_{\xi,v}:=\pi_v^*\cL_\xi$, $\xi\in X(\pi_0(T_v))$,
where 
$\pi_v:\cO_v=I_0\cdot n_v\cong I_0/I_0\cap(LG)^{\psi_v}
\rightarrow T_0/T_v$.

For the rest of the paper, 
we will fix a dimension theory $\delta$ on $L^\theta X$
which exists by Proposition \ref{p:dim theory}.
Denote by 
$D(I_0\backslash L^\theta X)$
the $I_0$-equivaraint derived category of $L^\theta X$
and 
$\on{Perv}_\delta(I_0\backslash L^\theta X)$
the corresponding category of perverse sheaves as in \cite[Appendix A]{CYUntwisted}.
We define the affine Hecke action 
of the affine Hecke category $D(I_0\backslash LG/I_0)$
and Verdier duality $\mathbb D_\delta$ 
on $D(I_0\backslash L^\theta X)$ 
in the same way as the untwisted case \cite[\S8.1]{CYUntwisted}.
We also similarly define abelian categories of Weil sheaves
$\on{Shv}(I_0\backslash LG/I_0)^{\on{Weil}}$,
$\on{Shv}(I_0\backslash L^\theta X)^{\on{Weil}}$,
$\on{Perv}(I_0\backslash LG/I_0)^{\on{Weil}}$
and $\on{Perv}_\delta(I_0\backslash L^\theta X)^{\on{Weil}}$,
and denote the Grothendieck groups of their subcategories of
objects with Frobenius eigenvalues in $q^{\frac{\bZ}{2}}$
up to roots of unity by
$H,M,H',M'$ respectively.
The convolution action makes $M$ into a $H$-module 
(resp. $M'$ into a $H'$-module).
The Verdier duality induces involutions $D$ on $H,H'$
and $D_\delta$ on $M,M'$.

For any element $w$ of the extended Weyl group $\widetilde{W}$,
let $j_w:LG_w=I_0wI_0\rightarrow LG$ be embedding.
For constant local systems $\cL_w$ on $LG_w$,
$j_{w,!}\cL_w$ define a $\bZ[q^{\frac{1}{2}},q^{-\frac{1}{2}}]$-basis $[\cL_w]$ of $H$,
and their IC-extensions define a $\bZ[q^{\frac{1}{2}},q^{-\frac{1}{2}}]$-basis $[\IC_w]$ of $H'$.
Similarly, for $I_0$-orbits $j_v:\cO_v\rightarrow L^\theta X$,
the rank one $I_0$-equivariant local systems $\cL_{\xi,v}$
define a $\bZ[q^{\frac{1}{2}},q^{-\frac{1}{2}}]$-basis 
$[\cL_{\xi,v}]$ of $M$,
and their IC-extensions 
$\IC_{\xi,v}=j_{v,!*}\cL_{\xi,v}[\delta(v)]$
define a $\bZ[q^{\frac{1}{2}},q^{-\frac{1}{2}}]$-basis
$[\IC_{\xi,v}]$ of $M'$.
The map $[\cA]\to \sum_{i\in\mathbb Z} (-1)^i[\sH^i(\cA)]$
induces bijections $H'\cong H$ and $M'\cong M$.

\subsubsection{}
Let $\alpha$ be a simple affine root and $\cO_v$ an $I_0$-orbit.
In the case of $\cO_v$ being of type IIIb or IVb for $s_\alpha$,
let $\cO_{v'}\neq \cO_v$ be the closed orbit in $P_\alpha n_v$
as in \cite[Lemma 28]{CYUntwisted}. 
Define the map $a_{v'}:
\hat{X}(T_0/T_v)\rightarrow
\hat{X}(T_0/T_v\cap\ker(\alpha))\rightarrow
\hat{X}(\bGm)$ as in \S8.3 of \emph{loc. cit.},
where $\hat{X}(T)=X^*(T)\otimes_\bZ(\bZ_{(p)}/\bZ)$ 
for any torus $T$.

The action of $[\cL_{s_\alpha}]\in H$
on $M$ is given by the same formula as 
\cite[Proposition 74]{CYUntwisted}.
The action of $[\cL_w]$ for $w\in\Omega$ is given by
the same formula as Lemma 75 in \emph{loc. cit.},

\subsection{The polynomials $b_{\eta,u;\xi,v}$ and $c_{\eta,u;\xi,v,i}$}\label{polynomials b and c}

Consider the partial ordering on $I_0$-orbits by closure relation
and write $u<v$ if $\cO_{u}\subset\overline{\cO}_{v}$
and $u\neq v$.

We have 
$\mathbb D_\delta(\IC_{\xi,v})\cong\IC_{-\xi,v}(\delta(v))$
and $\mathbb D_\delta(j_{v,!}\cL_{\xi,v})\cong j_{v,*}\cL_{-\xi,v}[2\delta(v)](\delta(v))$.
It follows that
\begin{equation}\label{eq:b coeff}
	D_\delta[\cL_{\xi,v}]=q^{-\delta(v)}[\cL_{-\xi,v}]+
	\sum_{u<v}b_{\eta,u;\xi,v}[\cL_{\eta,u}]\in M
\end{equation}
and 
\begin{equation}\label{eq:c coeff}
	[\sH^i\IC_{\xi,v}]=\delta_{i,-\delta(v)}[\cL_{\xi,v}]+
	\sum_{u<v}c_{\eta,u;\xi,v,i}[\cL_{\eta,u}]\in M
\end{equation}
where $b_{\eta,u;\xi,v},c_{\eta,u;\xi,v,i}\in\mathbb Z[q^{\frac{1}{2}},q^{-\frac{1}{2}}]$.
The polynomials $b_{\eta,u;\xi,v}$ can be viewed as 
another affine analogue of the $R_{\gamma,\delta}$ in \cite{LV}
beside the untwisted version,
and the $c_{\eta,u;\xi,v,i}$ are related to the 
twisted affine Kazhdan-Lusztig-Vogan polynomials in 
Section \ref{AKLV}.

Apply the bijection $h:M'\cong M$ to equality
$D_\delta[\IC_{\xi,v}]=q^{-\delta(v)}[\IC_{-\xi,v}]$
and plug in the relations \eqref{eq:b coeff}, \eqref{eq:c coeff}.
By comparing coefficients of $[\cL_{\eta,u}]$ for $u<v$, we obtain
\begin{equation}\label{eq:b c coeff relation}
	\begin{split}
		&\sum_i(-1)^ic_{\eta,u;-\xi,v,i}
		-q^{\delta(v)-\delta(u)}
		\sum_i(-1)^i\bar{c}_{-\eta,u;\xi,v,i}\\
		&=(-1)^{\delta(v)}q^{\delta(v)}b_{\eta,u;\xi,v}
		+q^{\delta(v)}\sum_{u<z<v}b_{\eta,u;\zeta,z}
		\sum_i(-1)^i\bar{c}_{\zeta,z;\xi,v,i}
	\end{split}
\end{equation}
where 
$\bar{c}\in\mathbb Z[q^{\frac{1}{2}},q^{\frac{-1}{2}}]$ means inverting $q^{\frac{1}{2}}$.

By the same proof as \cite[Lemma 76]{CYUntwisted}, we have
\begin{lem}\label{l:c coeff}
Assume $c_{\eta,u;\xi,v,i}\in\bN q^{\frac{1}{2}(i+\delta(v))}$.
	Then the coefficients $c_{\eta,u;\xi,v,i}$'s
		can be expressed as a product $a\prod_jb_j$ 
		where $a\in\bZ q^\bZ$,
		$b_j$ or $b_j^{-1}$ is a monomial of $q^{\frac{1}{2}}$
		in some $b_{\eta,u;\xi,v}$.
\end{lem}

\subsection{An algorithm}\label{ss:algorithm}
We have a parallel algorithm that computes $b_{\eta,u;\xi,v}$
as in the untwisted case \cite[\S8.5]{CYUntwisted}.
Since the arguments are the same as the untwisted case,
we only review the algorithm.
By Lemma \ref{l:c coeff} and Theorem \ref{purity} to be proved later,
this also provides a way to compute $c_{\eta,u;\xi,v,i}$
using \eqref{eq:b c coeff relation}.

Let $\cO_v$ be an $I_0$-orbit,
$\xi\in X(\pi_0(T_0))$.
If $\cO_v=\overline{\cO}_v$ is closed,
clearly $b_{\eta,u;\xi,v}=0$.
Otherwise by Lemma 31 of \emph{loc. cit.},
there exists simple affine root $\alpha$
such that $\cO_v$ is of type b for $s_\alpha$.
We can take induction on $\dim\cO_v$.

Case 1: $\cO_v$ is of type IIb for some $s_\alpha$.
By Lemma 28 of \emph{loc. cit.},
$P_\alpha\cO_v=\cO_v\sqcup\cO_{v'}$ 
in which $\cO_v$ is open.
As in \S8.5.1 of \emph{loc. cit.}, we have
\begin{equation}\label{eq:IIb induction}
	D[\cL_{\xi,v}]=(q^{-1}[\cL_{s_\alpha}]+q^{-1}-1)D[\cL_{s\xi,v'}].
\end{equation}
Using induction hypothesis and Proposition 74 of \emph{loc. cit.},
we can compute $b_{\eta,u;\xi,v}$ for any $u<v$.

Case 2: $\cO_v$ is not of type IIb for any $s_\alpha$.
Let $J$ be the set of simple affine roots such that
$\cO_v$ is of type IIIb or IVb for $s_\alpha$, $\alpha\in J$.
As in \S8.5.2 of \emph{loc. cit.},
$J$ is a proper subset of simple affine roots,
and the associated standard parahoric subgroup $P=P_J$
is $\psi_v$-stable with 
$\psi_v$-stable pro-unipotent radical $P^+$ 
and Levi subgroup $L_P$. 
By the same proof as Lemma 77 of \emph{loc. cit.},
we have $P\cdot n_v=\overline{\cO}_v$, 
so that we can define a quotient map
\[
\pi:\overline{\cO}_v=P\cdot n_v
\cong P/P^{\psi_v}\rightarrow L_P/L_P^{\psi_v}.
\]
It induces a bijection between $I_0$-orbits on $\overline{\cO}_v$
to the orbits of $B_{L_P}:=I_0\cap L_P$ on $X_{L_P}:=L_P/L_P^{\psi_v}$. 
Here $X_{L_P}$ is a finite dimensional symmetric variety,
whose set of $B_{L_P}$-orbits are parametrized by 
$\{\pi u\mid u\leq v\}$.
Denote by $j_{\pi u}:\cO_{\pi u}\hookrightarrow X_{L_P}$ 
the $B_{L_P}$-orbit corresponding to $\pi u$, 
and $\delta(\pi u)=\dim\cO_{\pi u}$.
By the same proof as Lemma 78 of \emph{loc. cit.} 
we have the following reduction to the finite dimensional case:
\begin{lem}\label{reduction to finite case}
	\begin{itemize}
		\item [(i)]
		$\pi^*\cL_{\eta,\pi u}=\cL_{\eta,u}$,\ $u\leq v$, 
            $\eta\in\hat{X}(T_0/T_u)$,
		
		\item [(ii)]
		$\IC_{\xi,v}\cong\pi^*\IC_{\xi,\pi v}[\delta(v)-\delta(\pi v)]$,
		
		\item [(iii)]
		$\mathbb D_\delta (j_{v,!}\cL_{\xi,v})\cong\pi^*(\mathbb D(j_{\pi v,!}\cL_{\xi,\pi v}))\langle2\delta(v)-2\delta(\pi v)\rangle$.
	\end{itemize}
\end{lem}

In view of the above, 
the computation of $b_{\eta,u;\xi,v},c_{\eta,u;\xi,v,i}$
can be reduced to the finite dimensional symmetric variety $X_{L_P}$,
which has been done in \cite{MS}.
By the same proof of Lemma 79 of \emph{loc. cit.}, we obtain
\begin{lem}\label{degree bound}
	$q^{\delta(v)}b_{\eta,u;\xi,v}$
        is a polynomial in $q$ with coefficients in $\mathbb Z$ of degree at most $\delta(v)-\delta(u)$.
\end{lem}

\subsubsection{Applications to spherical orbits}\label{sss:G(O) C coeff}
As in the untwisted case, 
the above results for Iwahoric orbits 
also provide an algorithm to compute  
 the multiplicities for $L^+G$-equivariant $\IC$-complexes.
Precisely,
let $\cL_{\chi,\nu}$ be a $L^+G$-equivariant local system 
on $L^+G$-orbit $j_\nu:LX_\nu\hookrightarrow L^\theta X$,
where $\chi$ is a representation of the component group of stabilizers
in $L^+G$ on $L^\theta X_\nu$.
Consider the associated $\IC$-complex
\[
\IC_{\chi,\nu}:=(j_{\nu})_{!*}\cL_{\chi,\nu}[\delta(\nu)]\in\on{Perv}_\delta(L^+G\backslash L^\theta X).
\]
We can write down similar formula as \eqref{eq:c coeff}
in the Grothendieck group of
mixed $L^+G$-constructible sheaves:
\begin{equation}\label{eq:C coeff}
	[\sH^i\IC_{\chi,\nu}]=\delta_{i,-\delta(\nu)}[\cL_{\chi,\nu}]+
	\sum_{\nu'<\nu}C_{\chi',\nu';\chi,\nu,i}[\cL_{\chi',\nu'}]
\end{equation}
where $C_{\chi',\nu';\chi,\nu,i}\in\mathbb Z[q^{\frac{1}{2}},q^{\frac{-1}{2}}]$,
and each $\cL_{\chi',\nu'}$ is a $L^+G$-equivariant local system on $L^\theta X_\nu$.

Let $h$ be the map from the Grothendieck group of 
mixed $L^+G$-constructible sheaves
into that of mixed $I_0$-constructible sheaves and
denote $h[\cL_{\chi',\nu'}]=\sum_{(\eta,u)\subset(\chi',\nu')}[\cL_{\eta,u}]$,
where $(\eta,u)\subset(\chi',\nu')$ means
$\cO_u\subset\cO_{\nu'}$ and $\cL_{\chi',\nu'}|_{\cO_u}=\cL_{\eta,u}$.
Denote by $\cO_{v_\nu}$ be unique open $I_0$-orbit in $L^\theta X_\nu$
and let $\cL_{\xi_\chi,v_\nu}=\cL_{\chi,\nu}|_{\cO_{v_\nu}}$.
Then 
\begin{equation}\label{IC=IC}
   \IC_{\chi,\nu}=\IC_{\xi_\chi,v_\nu}. 
\end{equation}
By comparing \eqref{eq:c coeff}
with the image under $h$ of \eqref{eq:C coeff},
we see their left hand sides are the same.
The matching of coefficients on the right hand sides provides relation
\begin{equation}\label{eq:c v.s. C}
	c_{\eta,u;\xi_\chi,v_\nu,i}=
	\begin{cases}
		\delta_{i,-\delta(\nu)},\hspace{1.2cm}(\eta,u)\subset(\chi,\nu),\\
		C_{\chi',\nu';\chi,\nu,i},\qquad (\eta,u)\subset(\chi',\nu'),\ \nu'<\nu,\\
		0,\hspace{2.3cm}\text{otherwise}.
	\end{cases}
\end{equation}

\section{Parity vanishing and Poincar\'e polynomials of IC-complexes
}\label{main results}
In this section, we prove the parity vanishing of 
equivariant $\IC$-complexes
on the spherical and Iwahori orbits.
We will first prove for Iwahori orbits
following the strategy of Mars-Springer \cite{MS} 
in the finite dimensional situation,
then deduce the result for spherical orbits.
Along the way, we provide an algorithm 
that computes the Poincar\'e polynomials of 
Iwahori-equivariant $\IC$-complexes
in an inductive procedure.
We keep the same notations and assumptions as 
in \S\ref{Affine Hecke modules}.

\subsection{Purity and parity vanishing}
By the same discussion as in \cite[\S9.1]{CYUntwisted},
for an $I_0$-equivariant local system $\cL_{\chi,v}$ 
on an $I_0$-orbit $\cO_v$,
the IC-complex $\IC_{\xi,v}$ has a canonical Weil structure 
$\Phi=\{\Phi_n,\ \text{for for $n$ divisible by $n_0$}\}$, 
where $\Phi_n:(F^n)^*\IC_{\xi,v}\cong\IC_{\xi,v}$
are systems of isomorphisms  
satisfying $(\Phi_n)^m=\Phi_{nm}$.
We also define the notion of 
$*$-pointwise pure (resp. $!$-pointwise pure)
in the same way as in \emph{loc. cit.}.
We have similar notions of Weil structure and pointwise purity for $\IC_{\chi,\nu}\in\on{Perve}_\delta(G(\cO)\backslash LX)^{{\on{Weil}}}$.

By the same proof as Theorem 81 of \emph{loc. cit.},
we have the following key pointwise purity result.
\begin{thm}\label{purity}
The complex $\IC_{\xi,v}$ (resp. $\IC_{\chi,\nu}$) is pointwise pure of weight $\delta(v)$ (resp. $\delta(\nu)$). 
\end{thm}

We are ready to prove the parity vanishing of $\IC$-complexes.
\begin{thm}\label{t:IC parity vanishing}
	For any $(\xi,v), (\eta,u)\in\sD$
    or $\nu\in\sN$ with character $\chi$, we have 
	\begin{itemize}
		\item [(i)] 
		$c_{\eta,u;\xi,v,i}\in\bN q^{\frac{1}{2}(i+\delta(v))}$.
		Moreover, $c_{\eta,u;\xi,v,i}=0$ 
		if $i+\delta(v)$ is odd.
		
		\item [(ii)]
		$\sH^i\IC_{\xi,v}=0$ if $i+\delta(v)$ is odd.

        \item[(iii)]
        $\sH^i\IC_{\chi,\nu}=0$ if $i+\delta(\lambda)$ is odd.
	\end{itemize}
\end{thm}
\begin{proof}
	(i):
	The first claim follows from the $*$-pointwise purity of $\IC_{\xi,v}$ in Theorem \ref{purity}.
	The second vanishing property follows from 
	Lemma \ref{l:c coeff} and Lemma \ref{degree bound}.

	(ii):
	This follows from part (i) and the definition \eqref{eq:c coeff}.

    (iii):
    It follows from (ii) and~\eqref{IC=IC}.	
\end{proof}

\subsection{Twisted affine Kazhdan-Lusztig-Vogan polynomials}\label{AKLV}
For $I_0$-equivariant local systems 
$\cL_{\eta,u}$ and $\cL_{\xi,v}$,
define Poincar\'e polynomial
\begin{equation}\label{eq:Iwahori Poincare poly}	P_{\eta,u;\xi,v}:=\sum_i[\cL_{\eta,u},\sH^i(j_{v,!*}\cL_{\xi,v})|_{\cO_u}]q^{\frac{i}{2}}.
\end{equation}
By Theorem \ref{t:IC parity vanishing}.(i), 
$[\cL_{\eta,u},\sH^i(j_{v,!*}\cL_{\xi,v})|_{\cO_u}]q^{\frac{i}{2}}
=c_{\eta,u;\xi,v,i-\delta(v)}$.
Thus
\[
P_{\eta,u;\xi,v}=\sum_ic_{\eta,u;\xi,v,i-\delta(v)}.
\]
Here we take convention that 
$c_{\xi,v;\xi,v,i}=\delta_{i,-\delta(v)}$, 
and $c_{\eta,u;\xi,v,i}=0$ if
$u$ is not in the closure of $\cO_v$
or $u$ is in the closure of $\cO_v$
but $\cL_{\eta,u}$ is not a subquotient of $\cL_{\xi,v}|_{\cO_{\bar{u}}}$.
Then
\[
h[j_{v,!*}\cL_{\xi,v}]=\sum_{\eta,u}P_{\eta,u;\xi,v}[\cL_{\eta,u}].
\]

By the same proof as \cite[Theorem 84]{CYUntwisted}, we have
\begin{thm}\label{t:I_0 Poincare poly uniqueness}
	For any pair $(\eta,u),(\xi,v)$, 
	$P_{\eta,u;\xi,v}$ is a polynomial in $q$ 
	with non-negative integer coefficients.
	It is the unique family of polynomials in $q^{\frac{1}{2}}$ 
	satisfying the following conditions:
	\begin{itemize}
		\item [(i)] $P_{\xi,v;\xi,v}=1$.
		\item [(ii)] If $\cO_u\neq\cO_v$, 
		$\deg_q(P_{\eta,u;\xi,v})\leq\frac{1}{2}(\delta(v)-\delta(u)-1)$.
		\item [(iii)] $C_{\xi,v}:=\sum_{\eta,u}P_{\eta,u;\xi,v}[\cL_{\eta,u}]$
		satisfies $D_\delta C_{\xi,v}=q^{-\delta(v)}C_{-\xi,v}$.
	\end{itemize}
\end{thm}

\subsection{Twisted Relative 
Kostka-Foulkes polynomials}\label{Def of KF poly}
We can deduce similar results for $L^+G$-orbits from the above.
We resume the discussion in \S\ref{sss:G(O) C coeff}.
For $L^+G$-equivariant local systems 
$\cL_{\chi,\nu}$ and $\cL_{\chi',\nu'}$,
define $P_{\chi',\nu';\chi,\nu}$ similarly as $P_{\eta,u;\xi,v}$.
By the same argument as the Iwahori case and \cite[Theorem 86]{CYUntwisted}, we obtain:
\begin{thm}\label{t:G(O) Poincare poly uniqueness}
	For any pairs $(\chi',\nu'),(\chi,\nu)$, 
	$P_{\chi',\nu';\chi,\nu}$ is a polynomial in $q$ 
	with non-negative integer coefficients.
	It is the unique family of polynomials in $q^{\frac{1}{2}}$ 
	satisfying the following conditions:
	\begin{itemize}
		\item [(i)] $P_{\chi,\nu;\chi,\nu}=1$.
		\item [(ii)] If $\nu'\neq\nu$, 
		$\deg_q(P_{\chi',\nu';\chi,\nu})\leq\frac{1}{2}(\delta(\nu)-\delta(\nu')-1)$.
		\item [(iii)] $C_{\chi,\nu}:=\sum_{\chi',\nu'}P_{\chi',\nu';\chi,\nu}[\cL_{\chi',\nu'}]$
		satisfies $D_\delta C_{\chi,\nu}=q^{-\delta(\nu)}C_{-\chi,\nu}$.
	\end{itemize}
\end{thm}

\section{Applications}\label{applications}

\subsection{Formality}
Let $D^{}(L^+G\backslash\Gr_G)$ be the derived Satake category for 
$G$ and 
let $D^{}(L^+G\backslash L^\theta X)$ be the 
derived relative Satake category of 
dg derived category of 
$L^+G$-equivariant complexes on $LX$.
We have the Hecke action, denoted by $\star$, of the  Satake category $D^{}(L^+G\backslash LG/L^+G)$ 
on $D^{}(L^+G\backslash L^\theta X)$.
We have the monoidal abelian Satake equivalence 
\[\on{Rep}(\check G)\cong\on{Perv}^{}(L^+G\backslash\Gr_G):\ \ \ V\to \ \IC_V.\]
By restricting the action to $\on{Perv}^{}(L^+G\backslash\Gr_G)\cong\on{Rep}(\check G)$, we obtain 
a monoidal action of $\on{Rep}(\check G)$ on $D^{}(L^+G\backslash L^\theta X)$.
Let $\mathrm{e}\in D^{}(L^+G\backslash L^\theta X)$ be the constant sheaf on the integral points
$L^\theta X\cap L^+G$, which is a finite union of closed $L^+G$-orbits, see \cite[Corollary 23]{CYMatsuki}.
Let $\IC_{\cO(\check G)}$
(an ind-object in $\on{Perv}^{}(L^+G\backslash\Gr_G)$)
be the image of the regular representation 
$\cO(\check G)$
under the abelian Satake equivalence.
Since $\cO(\check G)$ is a ring object in $\on{Rep}(\check G)$, the 
RHom space 
\begin{equation}\label{def of A_X}
A_{X,\theta}:=R\Hom_{D(L^+G\backslash L^\theta X)}(\mathrm{e}_{},\IC_{\cO(\check G)}\star\mathrm{e}_{})
\end{equation}
is naturally a dg-algebra, known as the de-equivariantized Ext algebra.

\begin{thm}\label{conj}
	The dg-algebra $A_{X,\theta}$
	is formal, that is, it is quasi-isomorphic to the 
	graded algebra 
	$A_{X,\theta}\cong\on{Ext}^\bullet_{D(L^+G\backslash L^\theta X)}(\mathrm{e},\IC_{\cO(\check G)}\star\mathrm{e}_{})$
	with trivial differential.
\end{thm}

\begin{proof}
	The same proof as in the untwisted case \cite[Theorem 93]{CYUntwisted} using the pointwise purity in Theorem \ref{purity}.
 \end{proof}

\subsection{Semisimplicity}
In the untwisted case, the parity vanishing of $\IC$-complexes on spherical orbit closures yields a semisimplicity criterion for the Satake category; see \cite[Theorem 87(i)]{CYUntwisted}. In the present twisted setting, the symplectic structures on transversal slices to spherical orbits yield the following stronger result:

\begin{thm}\label{p:semisimple}
	The relative Satake category
	$\mathrm{Perv}_\delta(L^+G\backslash L^\theta X)$ is semisimple.
\end{thm}
\begin{proof}
The symplecticity of the slices implies that spherical orbits have even codimension. Combining this with the parity vanishing in Theorem \ref{t:IC parity vanishing}, we obtain the claim by the general argument of \cite[Proposition 5.1.1]{ZhuIntroGr}; see also \cite[Theorem 87(i)]{CYUntwisted}.
\end{proof}

\subsection{Geometric Satake equivalence}
We recall some constructions and results in \cite[Section 6 and Section 7]{CYTempiric}.
We fix a pinned involution $\sigma_0$ of $G$
in the inner class of $\theta_0$
and consider the dual pinned involution $\check\sigma=-w_0\circ\check\sigma_0$
for the complex dual group $\check G$.
Let $I_{2\check\rho(-1)}$ be the equivalence classes of pure involutions 
of $\check G$ in the inner class of $\check\sigma$ as in \cite[Definition 14]{CYTempiric}.
We fix a set of representatives $\check\theta_i,i\in I_{2\check\rho(-1)}$
and let $\check K_i=\check G^{\check\theta_i}$ be the 
associated symmetric subgroups.
Let $\check G^{alg}\to\check G$ be the pro-algebraic group of finite coverings of $\check G$.
We have an exact sequence
\[1\ra\pi_1(\check G)^{alg}\to\check G^{alg}\to\check G\to 1\]
where $\pi_1(\check G)^{alg}$ is the projective limit of finite quotients of the fundamental group $\pi_1(\check G)$. Denote by $\check K^{alg}_i=\check G^{alg}\times_{\check G}\check K_i$ the base change of $\check K_i$.

Let $Z(G)$ be the center of $G$ 
and $Z(G)^{\sigma_0}_{tor}\subset Z(G)$ the subgroup of 
$\sigma_0$-fixed torsion elements. 
For any $z\in Z(G)^{\sigma_0}_{tor} $, we can form the 
extended loop symmetric space 
\[L^\sigma_zX=\{\gamma\in LG|\gamma\sigma(\gamma)=z^{-1}\}\]
with $LG$-action given by the $\sigma$-conjugation.
On the other hand,
via the duality \[\on{Hom}_{cont}(\pi_1(\check G)^{alg},\bC^\times)\cong Z(G)_{tor}\]
any such $z$ gives rise to a 
continuous character
$\chi_z:\pi_1(\check G)^{alg}\to\mathbb C^\times$,
and we set $\on{Rep}(\check K_i^{alg})_z$ to be the category 
of finite dimensional complex representation of $\check K_i$ such that the central subgroup $\pi_1(\check G)^{alg}$ acts by $\chi_a$. Let $\on{Irr}(\check K_i^{alg})_z$ be the set of irreducible representations in $\on{Rep}(\check K_i^{alg})_z$.
Note that if $z=1$ is trivial, for example when $\theta_0=\sigma_0$ is a pinned involution, then
$\chi_z$ is the trivial character and 
we have 
$\on{Rep}(\check K_i^{alg})_z=\on{Rep}(\check K_i)$.
Using the Matsuki duality for loop groups in \cite{CYMatsuki} and 
Vogan's theory of minimal $K$-types, it is shown in \cite[Theorem 19]{CYTempiric} that there is a natural bijection 
\[
\on{Loc}_{L^+G}(L^\sigma_z X)\longleftrightarrow\bigsqcup_{i\in I_{2\check\rho(-1)}}\on{Irr}(\check K^{alg}_i)_z,
\]
where the left hand side is the set of $L^+G$-equivariant irreducible local systems on spherical orbits on the $L^\sigma_z X$.
Note that $\theta_0$ is of the form $\theta_0=\Ad_a\circ\sigma_0$
for some $a\in G$, satisfying $a\sigma_0(a):=z^{-1}\in Z(G)_{tor}^{\sigma_0}$, and 
there is an $LG$-equivariant isomorphism
\[L^\theta X\cong L^\sigma_zX, \ \gamma\to \gamma a.\]
Thus we conclude that there is a natural bijection
\begin{equation}\label{bijection}
    \on{Loc}_{L^+G}(L^\theta X)\longleftrightarrow\bigsqcup_{i\in I_{2\check\rho(-1)}}\on{Irr}(\check K^{alg}_i)_z.
\end{equation}

Since the set of irreducible objects in
$\mathrm{Perv}_\delta(L^+G\backslash L^\theta X)$
are in bijection with $\on{Loc}_{L^+G}(L^\theta X)$
and the category $\on{Rep}(\check K^{alg}_i)_z$
is semisimple,  Theorem \ref{p:semisimple} implies:

\begin{thm}\label{Satake}
There is an equivalence of abelian categories
\[\mathrm{Perv}_\delta(L^+G\backslash L^\theta X)\cong\bigoplus_{i\in I_{2\check\rho(-1)}}\on{Rep}(\check K^{alg}_i)_z\]
such that the induced map between irreducible objects is given by the bijection in~\eqref{bijection}.
\end{thm}

\begin{exam}
Consider the case $G=\SL_2$, 
$\check G=\PGL_2$,
$\theta_0=\sigma_0=\on{id}$, $z=1$.
We have $L^\theta X(k)\cong\SL_2(F)/\SL_2(F')$
where $F'=k((t^2))$.
According to \cite[\S9]{CYTempiric},
there are two closed orbits, $\cO_0$ and $\cO_1$, each supporting only the trivial local system $\cL_{triv}$. The non-closed orbits are $\cO_r$, $r\geq 2$, and each $\cO_r$ supports two local systems, the trivial local system 
$\cL_{triv}$
and the sign local system $\cL_{sign}$.
We have $I_{2\check\rho(-1)}=\{1,2\}$ with 
$\check K_1=PO_2\cong\mathbb G_m\rtimes\{\pm1\}$ and $\check K_2=PGL_2$.
Let $W_r,r\geq0$ be the set of 
 irreducible representations of 
$PO_2$, where $W_0$ and $W_1$ are the trivial and sign representations and 
$W_r$ are the two dimensional irreducible representations such that $W_r|_{\mathbb G_m}=\chi_{r-1}\oplus\chi_{r-1}^{-1}$ ($\chi_{r-1}:\mathbb G_m\to\mathbb G_m, t\to t^{r-1}$). 
Let $V_{j}$ be the irreducible representation of $PGL_2$ of dimension $2j+1$.
Then Theorem \ref{Satake} takes the form
\[\mathrm{Perv}_\delta(L^+\SL_2\backslash L^\theta X)\cong \on{Rep}(PO_2)\oplus\on{Rep}(PGL_2)\]
\[\IC(\cO_r,\cL_{triv})\longleftrightarrow W_r,\ \ \ r\geq 0\]
\[\IC(\cO_r,\cL_{sign})\longleftrightarrow V_{r-2},\ \ \  r\geq 2.\]
\end{exam}

\end{document}